\documentclass[preprint,11pt]{imsart}

\usepackage{amsthm,amsmath,natbib, verbatim}
\RequirePackage[colorlinks,citecolor=blue,urlcolor=blue]{hyperref}

\usepackage{amssymb,graphicx,subfig}

\startlocaldefs
\newcommand{\MF}{\mathcal{F}}
\newcommand{\MC}{\mathcal{C}}

\newcommand{\cA}{\mathcal{A}}
\newcommand{\cP}{\mathcal{P}}
\newcommand{\Qro}{{\sf Q}}
\newcommand{\Pro}   {{\sf P}}
\newcommand{\Exp}  {{\sf E}}

\newcommand{\bN}{\mathbb{N}}

\newcommand{\Var}  {{\sf Var}}
\newcommand{\luName}{\text{gap-intersection }}
\newtheorem{theorem}{Theorem}[section]
\newtheorem{lemma}[theorem]{Lemma}
\newtheorem{corollary}[theorem]{Corollary}
\theoremstyle{remark}
\newtheorem{remark}{Remark}[section]

\endlocaldefs

\begin{document}

\begin{frontmatter}

\title{Asymptotically optimal, sequential, multiple testing procedures with  prior information on the number of signals}



\runtitle{Sequential Multiple Testing}

\begin{aug}
\author{\fnms{Y.} \snm{Song}\ead[label=e1]{ysong44@illinois.edu}}
\and
\author{\fnms{G.} \snm{Fellouris}\ead[label=e2]{fellouri@illinois.edu}}\\

\address{Department of Statistics, Coordinated Science Lab, \\ University of Illinois, Urbana-Champaign, \\ 725 S. Wright Street, Champaign 61820, USA\\ \printead{e1} \and \printead*{e2}}

\runauthor{Y. Song and G. Fellouris}
\end{aug}

\begin{abstract}
Assuming that data are collected sequentially from independent streams, we consider the  simultaneous testing of multiple binary hypotheses  under two  general setups; when the number of signals (correct alternatives) is known in advance, and when we only have a  lower and an upper bound for it.  In each of these setups,  we propose feasible  procedures that control, without any distributional assumptions, the familywise error probabilities of both type I and type II below  given, user-specified levels. Then, in the case of i.i.d. observations in each stream, we show that the proposed procedures achieve  the optimal  expected sample size,  under every  possible signal configuration, asymptotically as the two  error probabilities vanish at arbitrary rates.  A simulation study is presented in a completely  symmetric case and supports insights obtained from our asymptotic results, such as the fact that knowledge of the exact number of signals roughly halves the expected number of observations compared to  the case of no prior  information. 
\end{abstract}



\end{frontmatter}



\section{Introduction}
Multiple testing, that is the simultaneous consideration of $K$ hypothesis testing problems, $H_0^k$ versus $H_1^k$, $1 \leq k \leq K$,  is one of the oldest, yet still very active areas of statistical research. The vast majority of work in this area assumes a fixed set of observations  and focuses on testing procedures that control the  familywise  type I error (i.e., at least one false positive), as in \cite{marcus1976closed,holm1979simple,hommel1988stagewise},
or less stringent  metrics of this error, as in \cite{benjamini1995controlling} and \cite{lehmann2005}.



The multiple testing problem has been less studied  under the assumption that observations are acquired sequentially,  in which case the sample size is random.
The sequential setup is relevant in many applications, such as  multichannel signal detection~\citep{mei2008asymptotic,draglia1999multihypothesis},
outlier detection~\citep{li2014universal},  clinical trials with multiple end-points~\citep{bartroff2008generalized},
ultra high throughput mRNA sequencing data~\citep{bartroff_arxiv}, in which it is vital to make a quick decision in real time,  using the smallest possible number of observations. 

\citet{bartroff2010multistage} were the first to propose a sequential test  that controls the familywise error of type I.  \citet{de2012sequential,de2012step} and \citet{bartroff2014sequential} proposed universal sequential procedures that  control simultaneously the familywise errors of both type I \textit{and} type II, a  feature that is possible due to the sequential nature of sampling. The proposed sequential procedures in these works were shown through simulation studies to offer substantial savings in the average sample size in comparison to  the corresponding fixed-sample size tests.  

A  very relevant problem to multiple testing is the classification problem, in which there are $M$  hypotheses, $H_1, \ldots, H_M$, and the goal is to select the correct one among them. The classification problem has been studied extensively in the literature of sequential analysis, see e.g. 
\cite{sobel1949,armitage,
lorden1977nearly,tartakovsky1998asymptotic, draglia1999multihypothesis,dragalin2000multihypothesis}, generalizing the seminal  work of \citet{wald1945sequential} on  binary testing $(M=2)$. 
\citet{dragalin2000multihypothesis} considered the  multiple testing problem
as a special case of the classification problem  under the assumption of a \textit{single signal}  in $K$ independent streams,   and  focused on procedures  that control the probability of erroneously  claiming the signal to be in stream  $i$  for every  $1 \leq i \leq M=K$. 
In this framework, they proposed an asymptotically optimal sequential test as all these error probabilities go to 0.   The same approach of treating the multiple testing problem as a classification problem has been taken by~\citet{li2014universal} under the assumption of an upper bound on the number of  signals in the $K$ independent streams, and  a \textit{single  control}  on the  maximal  mis-classification probability.

We should stress that  interpreting multiple testing  as a  classification problem  does not  generally lead to feasible procedures. Consider, for example, the case of no prior information, which is the default assumption in the multiple testing literature. Then,  multiple testing  becomes a classification problem with $M=2^K$ categories and a brute-force implementation of existing classification procedures becomes  infeasible even for  moderate values of $K$, as the  number of statistics that need to be computed sequentially  grows exponentially with $K$.  Independently of feasibility considerations,  to the best of our knowledge there is no optimality theory regarding the expected sample size that can be achieved by  multiple testing procedures,  with or without prior information, that  control the  familywise errors of both  type I and type II. Filling this gap was one of the motivations  of this paper. 
 
The main contributions of the current work are the following: first of all, assuming that the data streams that correspond to the various hypotheses are independent,   we propose feasible  procedures that control  the familywise errors of both type I and type II  below  arbitrary, user-specified levels. We do so 
 under two general setups regarding prior information;  when the true number of signals  is known in advance,  and  when there is only a lower and an upper bound for it.   The former  setup  includes the case of a single  signal  considered in~\citet{draglia1999multihypothesis,dragalin2000multihypothesis}, whereas 
the latter includes  the case of no prior information, which is the underlying  assumption in~\citet{de2012sequential,de2012step,bartroff2014sequential}. While we provide universal threshold values that guarantee the desired error control in the spirit of the above works, we also propose a Monte Carlo simulation method based on importance sampling  for the efficient calculation of non-conservative thresholds in practice, even for very small error probabilities. 
More importantly,  in the case of independent and identically distributed (i.i.d.) observations in each stream,  we show  that   the proposed multiple testing procedures  attain  the optimal expected sample size, for \textit{any} possible signal configuration, to a first-order asymptotic approximation as the two  error probabilities go to zero in an \textit{arbitrary} way.  Our  asymptotic results also provide insights about the effect of prior information on the number of signals, which are corroborated by a simulation study.

The  remainder of the paper is organized as follows. In Section~\ref{formulation} we formulate the problem mathematically. In Section \ref{procedures} we present the proposed procedures and show how they can be designed to guarantee the desired error control. In  Section~\ref{is} we propose an efficient Monte Carlo simulation method for the determination of non-conservative critical values in practice.  In Section~\ref{theory} we establish the  asymptotic optimality  of the proposed procedures in the i.i.d.   setup.   In Section~\ref{simulation} we illustrate our asymptotic results with a simulation study. In Section~\ref{generalize} we conclude and  discuss potential generalizations of our work. Finally,  we present two useful lemmas for our proofs in an Appendix. 

\section{Problem formulation} \label{formulation}
Consider $K$ independent streams of observations, 
$X^k := \{X_n^k: n \in \bN\}$, $k \in [K]$,  where $[K] := \{1,\ldots, K\}$ and $\bN:= \{1, 2, \ldots\}$.  For each $k \in [K]$, let   $\Pro^{k}$ be the  distribution of   $X^k$, for which  we consider two simple hypotheses,
\begin{equation*}  
H_0^k:\; \Pro^k = \Pro_0^k \; \text{ versus }\; H_1^k:\; \Pro^k = \Pro_1^k, 
\end{equation*}
where $\Pro_0^k$ and $\Pro_1^k$ are distinct  probability measures  on the canonical space of $X^k$.  We will say that there is ``noise'' in the $k^{th}$ stream under 
$\Pro_0^k$ and ``signal'' under $\Pro_1^k$.   Our goal is to  simultaneously test these  $K$ hypotheses when data from all streams  become available sequentially and we want to make a  decision as soon as possible.

Let $\MF_n$ be the $\sigma$-field generated by all streams up to time $n$, i.e.,   $\MF_n = \sigma (X_1, \ldots, X_n )$, where   $X_n = ( X_n^1, \ldots ,X_n^K)$. We define  a  \textit{sequential} test for the multiple testing problem of interest to be  a  pair $(T,d)$ that  consists of an $\{\MF_n\}$-stopping time, $T$, at which we stop sampling in all streams,  and an $\MF_T$-measurable decision rule, $d=(d^1, \ldots, d^K)$, each component of which takes values in  $\{0,1\}$. The interpretation is that 
  we declare upon stopping   that there is signal (resp. noise) in the $k^{th}$ stream when $d^k=1$ (resp. $d^k=0$).  With an abuse of notation, we will also use $d$ to denote the  subset  of streams in which we declare that signal is present, i.e.,  $\{k \in [K] : \; d^k = 1\}$. 



For any subset $\cA \subset [K]$ we define the  probability measure
\begin{equation*} 
\Pro_{\cA} := \bigotimes_{k=1}^K \Pro^k ; \qquad\; \Pro^k = \begin{cases}
\Pro_0^k ,\qquad \text{ if } k \notin \cA \\
\Pro_1^k , \qquad \text{ if } k \in \cA
\end{cases} ,
\end{equation*}
such that  the distribution of $\{X_n, n \in \bN \}$  is $\Pro_\cA$ when $\cA$ is  the true subset of signals,
and  for an arbitrary sequential  test $(T,d)$  we set:
\begin{align*}
\{\cA \lesssim d \}  &:= \{ (d \setminus \cA) \neq \emptyset\} = \bigcup_{j \not \in \cA} \{d^{j} = 1\}, \\
\{d \lesssim \cA \}  &:=   \{(\cA \setminus d) \neq \emptyset\} =\bigcup_{k \in \cA} \{d^k = 0\}.
\end{align*}
Then,  $\Pro_{\cA}( \cA \lesssim d )$ is the probability  of at least one false positive (\textit{familywise type I error}) and $\Pro_{\cA} \left( d \lesssim \cA \right)$ the probability  of at least one false negative (\textit{familywise type II error}) of $(T,d)$  \textit{when the true subset of signals is $\cA$}. 

In this work we are  interested in sequential tests that  control these probabilities   below  user-specified levels  $\alpha$ and $\beta$ respectively, where $\alpha, \beta \in (0,1)$, for any possible  subset of signals.   In order to be able to incorporate  prior information,  we assume that the true subset of signals is known to  belong to a class  $\cP$ of subsets of $[K]$,  not necessarily equal to the  powerset, 
and we  focus on sequential tests in  the class
 $$
 \Delta_{\alpha,\beta}(\cP) := \left\{ (T,d):  \Pro_{\cA}(\cA \lesssim d) \leq \alpha \; \; \text{and} \;\;  \Pro_{\cA} \left( d \lesssim \cA \right) \leq \beta  \; \;  \text{for every} \; \;  \cA \in \cP \right\}.
 $$
 
We   consider, in particular, two general cases for class  $\cP$. In the  first one,   it is known that  there are  exactly $m$ signals in the $K$ streams, where   $1 \leq m \leq K-1$. In the second,  it is known that  there are at least $\ell$  and at most $u$ signals,  where   $0 \leq \ell < u \leq K$. In the former case we write  $\cP=\cP_{m}$ and in the latter $\cP=\cP_{\ell, u}$, where 
\begin{equation*} 
\cP_m:= \left\{ \cA \subset [K]: |\cA| = m \right\}, \quad \cP_{\ell, u}:= \left\{ \cA \subset [K]: \ell \leq  |\cA| \leq  u \right\}.
\end{equation*}
When $\ell=0$ and $u=K$, the class $\cP_{\ell, u}$ is the powerset  of  $[K]$, which corresponds to the  case of no prior information regarding the multiple testing problem.  

Our  main focus  is on  multiple testing procedures that not only belong to  $\Delta_{\alpha,\beta}(\cP)$ for a given class $\cP$, 
but also achieve the minimum possible  expected sample size, under each possible signal configuration, for small error probabilities. To be more specific, let  $\cP$ be a given class of subsets and  let   $(T^*,d^*)$ be a  sequential test that  can designed to  belong to $\Delta_{\alpha,\beta}(\cP)$  for any given $\alpha,\beta \in (0,1)$. We say that $(T^*,d^*)$  is \textit{asymptotically optimal with respect to class $\cP$}, if for every $\cA \in \cP$  we have  as $\alpha,\beta \to 0$
\begin{equation*}
\Exp_\cA \left[T^* \right] \sim \inf_{(T,d) \in \Delta_{\alpha,\beta}(\cP)} \Exp_\cA\left[T\right] ,
\end{equation*} 
where  $\Exp_{\cA}$ refers to  expectation under $\Pro_{\cA}$ and  $x\sim y$ means that $x/y \rightarrow 1$.   The ultimate goal of this work is to  propose feasible  sequential tests that are asymptotically optimal  with respect to  classes of the form $\cP_{m}$ and $\cP_{\ell, u}$.



\subsection{Assumptions and notations}

Before we continue with the presentation and analysis of the  proposed multiple testing procedures, we will introduce some additional notation, and impose  some minimal conditions on the distributions in each stream, which  we will assume to hold throughout the paper. 

First of all,  for each stream $k \in [K]$ and time $n \in \bN$ we assume that the probability measures $\Pro_0^k$ and $\Pro_1^k$ are  mutually absolutely continuous  when restricted to the $\sigma$-algebra $\MF_n^k=\sigma(X_{1}^k, \ldots, X_n^k)$,  and we denote by
\begin{equation} \label{llr}
\lambda^{k}(n) := \log \frac{d\Pro_{1}^k}{d\Pro_{0}^k}  (\MF_n^k)
\end{equation}
the cumulative log-likelihood ratio  at time $n$ based on the data in the $k^{th}$ stream.  Moreover, we assume that  for each stream $k \in [K]$ the probability measures $\Pro_0^k$ and $\Pro_1^k$ are singular on $\MF_\infty^k:=\sigma (\cup_{n \in \bN} \MF_n^k)$, which implies that   
\begin{align}\label{reg_lr}
\Pro_0^k \left(\lim_{n \to \infty} \lambda^k(n) = -\infty \right) 
= \Pro_1^{k} \left(\lim_{n \to \infty} \lambda^k(n) = \infty \right) = 1.
\end{align}
Intuitively, this means that as observations accumulate, the evidence in favor of the correct hypothesis  becomes arbitrarily strong. The latter assumption is  necessary in order to design procedures that terminate almost surely under every scenario. \textit{We do not make any other distributional assumption  until Section \ref{theory}}.


We use the   following notation for the ordered, local, log-likelihood ratio statistics at time $n$:
 $$\lambda^{(1)}(n) \geq  \ldots  \geq \lambda^{(K)}(n),$$
 and we denote by  $i_1(n), \ldots, i_K(n)$  the corresponding stream indices, 
 i.e., 
  $$\lambda^{(k)}(n) = \lambda^{i_k(n)}(n), \text{ for every } k \in [K].$$
   
Moreover, for every $n \in \bN$ we denote by  $p(n)$ the number of positive log-likelihood ratio statistics at time $n$, i.e., 
 $$\lambda^{(1)}(n) \geq  \ldots \geq \lambda^{(p(n))}(n) > 0 \geq \lambda^{(p(n) +1)}(n) \geq \ldots \geq \lambda^{(K)}(n).$$ 
 
 For any two subsets $\cA, \MC \subset [K]$ 
we denote by   $\lambda^{\cA,\MC}$ the log-likelihood ratio  process of $\Pro_{\cA}$ versus $\Pro_{\MC}$, i.e.,
\begin{align}\label{ll_diff}
\lambda^{\cA,\MC}(n) &:=  \log  \frac{d\Pro_{\cA}}{d\Pro_{\MC}}  (\MF_n) 
= \sum_{k \in \cA\setminus \MC} \lambda^{k}(n) -
\sum_{k \in \MC\setminus \cA} \lambda^{k}(n), \quad n \in \mathbb{N}.
\end{align}

Finally,  we use $|\cdot|$ to denote set cardinality, for any two real numbers $x,y$ we set $x\wedge y=\min\{x,y\}$ and $x\vee y=\max\{x,y\}$,   and for any measurable event $\Gamma$ and  random variable $Y$ we use the following notation
$$\Exp_{\cA}[ Y ; \Gamma] := \int_{\Gamma} Y d\Pro_\cA.$$

\section{Proposed sequential multiple testing procedures} \label{procedures}

In this section we present the proposed procedures and  show how they can be designed in order to guarantee the desired error control.

\subsection{Known number of signals}
In this subsection we consider the setup in which the number of signals  is known to be equal to $m$ for some $1 \leq m \leq K-1$, thus, $\cP = \cP_m$.  Without loss of generality, we  restrict ourselves to multiple testing procedures  $(T, d)$ such that  $|d|=m$. Thus,  the class of admissible sequential tests  takes the form 
$$
\Delta_{\alpha,\beta}(\cP_{m})=  \left\{ (T,d) :  \; \Pro_{\cA}(d\neq \cA) \leq \alpha \wedge \beta \, \; \text{for every} \; \cA \in \cP_{m} \right\},
$$
since  for  any  $\cA \in \cP_{m}$ and $(T, d)$ such that  $|d|=m$ we have 
$$
\{\cA \lesssim d \}\; = \; \{d \lesssim \cA \}\; = \; \{d\neq \cA\}.
$$
In this context, we propose the following sequential scheme:  stop as soon as  the \textit{gap} between the  $m$-th and $(m+1)$-th ordered  log-likelihood ratio statistics becomes larger than some constant  $c > 0$, and  declare that signal is present in the $m$ streams with the  top  log-likelihood ratios at the time of stopping. 
Formally, we propose the following procedure, 
to which we refer as ``gap rule":
\begin{align} \label{gap}
\begin{split} 
T_G &:= \inf \left\{n \geq 1: \lambda^{(m)}(n) - \lambda^{(m+1)}(n) \geq c\right\} ,  \\
d_G &:= \{i_1(T_G), \ldots, i_m(T_G)\}.
\end{split}
\end{align}
Here, we suppress the dependence of $(T_G, d_G)$ on $m$ and $c$  to lighten the notation.   The next theorem shows 
how to select threshold $c$ in order to guarantee the desired error control.

\begin{theorem} \label{Error_control_TG}
Suppose that assumption~\eqref{reg_lr} holds. Then, for any $\cA \in \cP_{m}$ and $c>0$  we have $\Pro_{A}(T_G < \infty) = 1$ and 
\begin{equation} \label{bound}
\Pro_\cA\left( d_G \neq \cA \right) \leq m(K-m) e^{-c}.
\end{equation}
Consequently,   $(T_G, d_G) \in \Delta_{\alpha,\beta}(\cP_m)$ when 
  threshold $c$ is selected as 
\begin{equation} \label{thres}
c = | \log (\alpha\wedge \beta) | +  \log ( m(K-m)). 
\end{equation} 
\end{theorem}

\begin{proof}

Fix $\cA \in \cP_m$ and $c>0$. We observe that   $T_G \leq T_G'$, where
\begin{align} \label{T_G_prime} 
\begin{split}
 T'_G &= \inf \left\{ n \geq 1: \lambda^{(m)}(n) - \lambda^{(m+1)}(n) \geq c , \; i_{1}(n) \in \cA, \ldots,  i_{m}(n) \in \cA  \right\}  \\ 
&= \inf \left\{n \geq 1: \lambda^{k}(n) - \lambda^{j}(n) \geq c \; \; \text{for every} \, k \in \cA \; \text{and} \; j \notin \cA \right\}.
\end{split}
 \end{align}  
Due to  condition~\eqref{reg_lr}, it is clear that $\Pro_\cA(T_G' < \infty) = 1$, which proves that $T_G$ is also almost surely finite under $\Pro_\cA$.  We now focus on proving \eqref{bound}.  The gap rule makes a mistake  under $\Pro_\cA$  if there exist 
 $ k \in \cA$ and $j \notin \cA$ such that the event $\Gamma_{k,j}= \left\{\lambda^{j}(T_G) -\lambda^{k}(T_G) \geq c \right\}$ occurs. In other words, 
$$ \left\{ d_G \neq \cA \right\} =\bigcup_{k \in \cA, j \notin \cA} \Gamma_{k,j},$$
and  from  Boole's inequality we have 
\begin{equation*}
\Pro_{\cA}( d_G \neq \cA ) \leq \sum_{k \in \cA, j \notin \cA} \Pro_{\cA}(\Gamma_{k,j}).
\end{equation*} 
 Fix $k \in \cA, j \notin \cA$ and  set $\MC = \cA\cup \{j\} \setminus \{k\}$.  Then,  from~\eqref{ll_diff} we have that  $\lambda^{\cA,\MC}= \lambda^{k}-\lambda^{j}$ and  from Wald's likelihood ratio identity   it follows that 
\begin{align} 
\begin{split}
\Pro_\cA(\Gamma_{k,j})  &= \Exp_{\MC}\left[\exp\{\lambda^{\cA, \MC}(T_G) 
\}; \Gamma_{k,j} \right] \\
&= \Exp_{\MC}\left[ \exp\{\lambda^{k}(T_G) - \lambda^{j}(T_G)\}; \Gamma_{k,j} \right]  \leq e^{-c},
\end{split}\label{com}
\end{align}
where the last inequality holds because $\lambda^{j}(T_G) - \lambda^{k}(T_G) \geq c$ on $\Gamma_{k,j}$. Since $|\cA| =m$ and $|\cA^c| = K -m$,   from the last two inequalities we obtain \eqref{bound}, which completes the proof.
\end{proof}

\subsection{Lower and upper bounds on the number of signals}
In this subsection,  we consider the setup in which we know  that there are at least $\ell$ and at most $u$ signals for some  $0\leq \ell < u \leq K$, that is, $\cP = \cP_{\ell,u}$.   In order to describe the proposed procedure, it is useful to first introduce the  ``intersection rule",  $(T_I,d_{I})$, according to which   we stop sampling as soon as \textit{all} log-likelihood ratio statistics are outside the interval $(-a,b)$, and at this time we  declare that signal is present (resp. absent) in those streams with positive (resp. negative) log-likelihood ratio, i.e., 
\begin{align}  \label{intersection}
\begin{split}
T_I &:= \inf \left\{n \geq 1: \lambda^k(n) \not \in (-a,b)  \;  \;  \text{for every} \; k \in [K] \right\},\\
d_I &:= \{i_{1}(T_I), \ldots, i_{p(T_{I})}(T_I)\},
\end{split} 
\end{align}
recalling  that  $p(n)$ is the number of positive log-likelihood ratios at time $n$.   This procedure was proposed by~\citet{de2012sequential}, where it was also shown that
when  the thresholds are selected as 
  \begin{equation} \label{thres-intersection}   
 a = |\log \beta|+ \log K , \quad b = |\log \alpha|+ \log K,
 \end{equation}
 the familywise type-I and type-II error probabilities are bounded by  $\alpha$ and $\beta$ for any possible signal configuration, i.e.,   $(T_I,d_{I}) \in \Delta_{\alpha,\beta}(\cP_{0,K})$.

A straightforward way to incorporate  the prior information of at least $\ell$ and at most $u$ signals   in the intersection rule is to  modify the stopping time  in \eqref{intersection} as follows:
\begin{align}\label{modifed-intersection}
\tau_2 &:= \inf\left\{n \geq 1: \ell \leq p(n) \leq u \; \text{and} \;\lambda^k(n) \not \in (-a,b) \; \text{for every} \; k \in [K]  \right\},
\end{align}
while keeping the same decision rule as in \eqref{intersection}.  Indeed,  stopping according to $\tau_2$  guarantees that the  number of null hypotheses  rejected upon stopping will be between $\ell$ and $u$. However, as we will see in Subsection \ref{noprior}, this rule will not in general  achieve asymptotic optimality in the boundary cases of exactly $\ell$ and  exactly  $u$ signals. In order to obtain an asymptotically optimal rule, 
we need to be able to stop faster when there are exactly $\ell$ or $u$  signals, which can  be achieved by stopping at 
\begin{align*}  
\tau_1 &:= \inf\left\{n \geq 1:   \lambda^{(\ell+1)}(n) \leq -a,\; \lambda^{(\ell)}(n) - \lambda^{(\ell+1)}(n) \geq c  \right\},\\
\quad \text{and} \quad \tau_3 &:= \inf \left\{ n \geq 1:  \lambda^{(u)}(n) \geq b,\;  \lambda^{(u)}(n) - \lambda^{(u+1)}(n) \geq d \right\},
\end{align*}
respectively. Here, $c$ and $d$ are additional positive thresholds that will be selected, together with $a$ and $b$, in order to guarantee the desired error control.

We can think of $\tau_1$ as a combination of the intersection rule and the gap rule that corresponds to the case of exactly $\ell$ signals. Indeed,  $\tau_1$ stops when  $K-\ell$  log-likelihood ratio statistics are simultaneously below $-a$,  but unlike the intersection rule  it  does not wait for the remaining $\ell$   statistics to be larger than $b$; instead,  similarly to the gap-rule in~\eqref{gap} with $m=\ell$, 
it requires the gap between the top $\ell$ and the bottom $K- \ell$ statistics to be larger than $c$.  In a similar way, $\tau_3$ is a combination of the intersection rule and the gap rule that corresponds to the case of exactly $u$ signals. 

Based on the above discussion,  when we know  that there are at least $\ell$ and at most $u$ signals, we  propose  the following procedure, to which  we  refer as ``gap-intersection'' rule:
 \begin{align}\label{gap-intersection}
T_{GI} &:= \min\{\tau_1,\tau_2,\tau_3\} , \quad d_{GI} := \{i_1(T_{GI}), \ldots,i_{p'}(T_{GI})\},
\end{align} 
where $p':=( p(T_{GI}) \wedge\ell ) \vee u$ is a truncated version of the number of positive log-likelihood ratios  at $T_{GI}$, i.e.,  if $p'=\ell$ when $p(T_{GI}) \leq \ell$, $p'=u$ when $p(T_{GI}) \geq u$ and  $p'=p(T_{GI})$ otherwise. In other words, 
we stop  sampling  as soon as one of the stopping  criterion in $\tau_1$, $\tau_2$ or $\tau_3$ is is satisfied, and  we reject upon stopping the null hypotheses in the $p'$ streams with the highest log-likelihood ratio values at time $T_{GI}$. 


As before, we suppress the dependence  on $\ell,u$ and $a,b,c,d$ in order to lighten the notation.  Moreover,  we set $\lambda^{(0)}(n) = -\infty$ and $\lambda^{(K+1)}(n) = \infty$ for every $n \in \bN$, which implies that if     $\ell = 0$, then $\tau_1 = \infty$, and if  $u = K$, then $\tau_3 = \infty$.  When in particular   $\ell = 0$ and $u =K$, that is the case of no prior information,  $T_{GI}=\tau_2$ and  $(T_{GI}, d_{GI})$  reduces to the intersection rule, $(T_{I}, d_{I})$, defined in  \eqref{intersection}. 

The following theorem   shows 
how  to select thresholds $a,b,c,d$ in order to guarantee the desired error control for the gap-intersection rule. 

\begin{theorem} \label{error_control_LU} 
Suppose that assumption~\eqref{reg_lr} holds. 
For any subset $\cA \in \cP_{\ell,u}$ and positive thresholds $a,b,c,d$, we have $\Pro_{A}(T_{GI} < \infty) = 1$ and 
\begin{align}   \label{bound2} 
\begin{split}
\Pro_\cA( \cA \lesssim d_{GI} )  &\, \leq  |\cA^c|  \, \left( e^{-b} +|\cA| \,   e^{-c} \right), \\
\Pro_\cA( d_{GI} \lesssim \cA )&\,\leq    |\cA| \,  \left(e^{-a} +  |\cA^c| \,  e^{-d} \right). 
\end{split}
\end{align}
In particular, $(T_{GI},d_{GI}) \in \Delta_{\alpha,\beta}(\cP_{\ell,u})$ when  the thresholds $a,b,c,d$ are selected as follows:
\begin{align} 
\begin{split}
a &= |\log \beta|+ \log K ,  \quad d = |\log \beta| + \log (u K), \\
b &= |\log \alpha|+ \log K, \quad  c =|\log \alpha| + \log ((K-\ell)K) .
\end{split} \label{thres_LU} 
\end{align}
\end{theorem}

\begin{proof}
Fix  $\cA \in \cP_{\ell,u}$ and $a,b,c,d>0$. Observe that $T_{GI} \leq \tau_2 \leq \tau_2'$, where 
\begin{align}\label{tau2_prime}
\tau_2' = \inf\{n \geq 1: -\lambda^{j}(n) \geq a,\; \lambda^{k}(n) \geq b \; \; \text{for every} \;  k \in \cA, j \notin \cA \}.
\end{align}
Due to assumption~\eqref{reg_lr},  $\Pro_\cA(\tau_2' < \infty) = 1$, which proves that $T_{GI}$ is also almost surely finite under $\Pro_\cA$. 
We now focus on proving the bound in \eqref{bound2} for the  familywise type-II error probability, since the corresponding result for the  familywise type-I error  can be shown similarly.   From Boole's inequality we have
\begin{align}   \label{error}
\Pro_\cA( d_{GI} \lesssim \cA )= \Pro_\cA \left( \bigcup_{k \in \cA} \{d_{GI}^{k} = 0\} \right) \leq \sum_{k \in \cA}  \Pro_\cA \left(d_{GI}^k =0 \right).
\end{align}
Fix $k \in \cA$.  Whenever the \luName rule  mistakenly accepts  $H_0^k$,  either the event $\Gamma_k:=\{\lambda^{k}(T_{GI}) \leq -a\}$ occurs (which is the case when stopping  at  $\tau_1$ or $\tau_2$), or  there is at least one $j \notin \cA$ such that the event   $\Gamma_{k,j}:=\{\lambda^{j}(T_{GI}) - \lambda^{k}(T_{GI})\geq d\}$ occurs (which is the case when stopping  at $\tau_3$). Therefore,
$$\{d_{GI}^k = 0\}\subset  \Gamma_{k}  \cup (\cup_{j \notin \cA} \Gamma_{k,j}),$$
and from Boole's inequality we have
\begin{align*}
\Pro_\cA(d_{GI}^k = 0 ) \leq  \Pro_\cA ( \Gamma_k )  +   \sum_{j \notin \cA} \Pro_\cA\left( \Gamma_{k,j} \right).
\end{align*}  
Identically to~\eqref{com} we can show that for every $j \notin \cA$ we have 
 $\Pro_\cA\left(\Gamma_{k,j}\right) \leq  e^{-d}.$  Moreover, 
if we set $ \MC= A\setminus \{k\}$ (note that $C \notin \cP_{\ell,u}$, but this does not affect our argument),  then $\lambda^{\cA, \MC}= \lambda^{k}$ and from Wald's likelihood ratio identity we have
\begin{align*}
\Pro_\cA ( \Gamma_k  )&= \Exp_\MC \left[ \exp\{\lambda^{\cA,\MC}(T_{GI})\}; \Gamma_k \right]  = \Exp_\MC \left[ \exp\{\lambda^{k}(T_{GI})\}; \Gamma_k \right]   \leq e^{-a}.
\end{align*}
Thus,
$$\Pro_\cA (d^{k}_{GI} = 0) \leq e^{-a} + (K - |\cA|) e^{-d},$$
 which together with \eqref{error} yields
\begin{align*}
\Pro_\cA( d_{GI} \lesssim \cA  )   &\, \leq    |\cA|   (e^{-a} +  |\cA^c|  e^{-d} )\leq  \frac{|\cA|}{K}  (K e^{-a}) +   \frac{|\cA^c|}{K}  (u K e^{-d}).
\end{align*}
Therefore,  if the thresholds are selected according to~\eqref{thres_LU}, then
$K e^{-a}=\beta$ and  $u K e^{-d}= \beta,$
which implies  that
$$\Pro_\cA( d_{GI} \lesssim \cA  )  \leq \frac{|\cA|}{K}  \beta +   \frac{|\cA^c|}{K}  \beta= \beta,
$$
and  the proof is complete. 
\end{proof}

\section{Computation of familywise error probabilities via importance sampling}\label{is}

The threshold specifications  in~\eqref{thres} and~\eqref{thres_LU} guarantee the desired  error control for the gap rule and gap-intersection rule respectively,  however they can be very conservative.
In practice, it is preferable to use  Monte Carlo simulation  to determine the thresholds that equate (at least, approximately)  the \emph{maximal} familywise type I and type II error probabilities to the corresponding target levels $\alpha$ and $\beta$, respectively. Note that this needs to be done offline, before the implementation of the procedure.  

When $\alpha$ and $\beta$ are very small, the corresponding errors are ``rare events'' and plain  Monte Carlo will not be efficient.  For this reason, in this section we propose 
a Monte Carlo approach based on  \textit{importance sampling}
for the efficient computation of the  familywise  error probabilities of the proposed multiple testing procedures.

To be more specific,  let $\cA \subset [K]$  be the true subset of signals and consider the  computation   of the familywise type I error probability,  $\Pro_\cA(\cA \lesssim d)$, of an arbitrary multiple testing procedure, 
$(T,d)$.  The idea of importance sampling is to find a  probability measure $\Pro_{\cA}^*$, under which the stopping time $T$ is finite almost surely,  and compute the desired probability by estimating (via plain Monte Carlo)  the expectation in the right-hand side of the following identity:
\begin{align*}
 \Pro_\cA( \cA \lesssim d) = \Exp^{*}_{\cA}\left[  (\Lambda^{*}_\cA)^{-1}; \cA \lesssim d\right],
\end{align*}
which is obtained by an  application of  Wald's likelihood ratio identity. Here,  we denote by  $\Lambda_{\cA}^{*}$ the likelihood ratio of $\Pro^{*}_{\cA}$ against $\Pro_\cA$ at time $T$, i.e., 
\begin{align*} 
\Lambda^*_\cA =  \frac{d \Pro^*_\cA}{d\Pro_\cA}(\MF_T),
\end{align*}
and by $\Exp^{*}_{\cA}$ the expectation under $\Pro^{*}_{\cA}$. The proposal distribution $\Pro^{*}_{\cA}$ should be selected such that $\Lambda^*_\cA$ is ``large'' on the event $\{\cA \lesssim d\}$ and ``small'' on its complement. This intuition will  guide us in the selection of  $\Pro^{*}_{\cA}$ for the proposed rules. 

For the gap rule $(T_G, d_G)$  we suggest the proposal distribution  to be a  uniform mixture  over  $\{\Pro_{\cA \cup \{j\} \setminus \{k\} }, k \in \cA,  j \notin \cA\}$, i.e., 
 \begin{align} \label{gap_measure}
\Pro^{G}_{\cA} := \frac{1}{|\cA|\,  |\cA^c|} \sum_{k \in \cA} \sum_{j \notin \cA} \Pro_{\cA \cup \{j\} \setminus \{k\}},
\end{align}
whose  likelihood ratio against $\Pro_\cA$  at time $T_G$ is 
\begin{align*}
\Lambda^G_\cA &:= \frac{1}{|\cA|  \, |\cA^c|} \sum_{k \in \cA} \sum_{j \notin \cA} \exp\{\lambda^{j}(T_G) - \lambda^{k}(T_G) \}.
\end{align*}
Then, on the event  $\{\cA \lesssim d_{G}\}$  there exists  some $k \in \cA$ and $j \notin \cA$ such that $\lambda^{j}(T_G) - \lambda^{k}(T_G) \geq c$, which leads to a large value for $\Lambda^G_\cA$.  On the other hand,  on the complement of $\{\cA \lesssim d_{G}\}$, $\{d_G = \cA\}$, 
we have  $\lambda^{j}(T_G) - \lambda^{k}(T_G) \leq - c$  for every $k \in \cA, j \notin \cA$,  which leads to a value  of  $\Lambda^{G}_\cA$ close to 0.

For the intersection rule $(T_I, d_I)$ we suggest the  proposal distribution to be a uniform mixture  over  $\{\Pro_{\cA \cup \{j\}},  j \notin \cA\}$, i.e., 
\begin{align} \label{intersection_measure}
\Pro_{\cA}^{I} := \frac{1}{ |\cA^c|} \sum_{j \notin \cA} \Pro_{\cA \cup \{j\}},
\end{align}
whose  likelihood ratio  against $\Pro_{\cA}$ at time $T_I$ takes the form
$$
\Lambda^{I}_\cA := \frac{1}{ |\cA^c|} \sum_{j \notin \cA}  \exp\{\lambda^{j}(T_I) \}.
$$
Note that on the event $\{\cA \lesssim d_I\}$ there exists some $j \notin A$ such that $\lambda^{j}(T_I) \geq b$, which results in a large  value for $\Lambda^{I}_\cA$. On the other hand, on the complement of $\{\cA \lesssim d_I\}$ 
we have  $\lambda^{j}(T_I) \leq -a$ for every $j \notin \cA$, which results in a  value of  $\Lambda^{I}_\cA$ close to 0.  

Finally, for the gap-intersection rule we suggest to use $\Pro_\cA^{I}$, the same proposal  distribution as in the intersection rule, when $\ell < |\cA| < u$. In the boundary case, i.e. $ |\cA|=\ell$ or $|\cA|= u$,  we propose the following  mixture of $\Pro_{\cA}^{G}$ and $\Pro_{\cA}^{I}$: 
\begin{align*}
\Pro^{GI}_{\cA} :=   \frac{|\cA|}{1+|\cA|} \,  \Pro_{\cA}^G + \frac{1}{1+|\cA|} \, \Pro_{\cA}^I. 
\end{align*}

In Section~\ref{simulation} we apply the  proposed simulation approach for the specification of non-conservative  thresholds in the case of identical, symmetric hypotheses with Gaussian i.i.d. data. 
We also refer to~\cite{simulation2016} for an analysis of these importance sampling estimators.  



\section{Asymptotic optimality in the i.i.d. setup} \label{theory}
From now on, we assume that, for  each stream  $k \in [K]$,   the observations $\{X_n^{k}, n \in \bN\}$  are  independent random variables with common density $f_i^k$ with  respect to  a $\sigma$-finite measure $\mu^k$ under $\Pro_i^k$, $i=0,1$,
 such that the  Kullback–--Leibler information numbers 
\begin{equation*}
D_0^k := \int \log\left(\frac{f_0^k}{f_1^k}\right) f_0^k d\mu^k, \quad
D_1^k := \int \log\left(\frac{f_1^k}{f_0^k}\right) f_1^k d\mu^k
\end{equation*}
are both positive and finite.   As a result,  for each $k \in [K]$  the  log-likelihood ratio process  in the $k^{th}$ stream, defined in \eqref{llr}, takes the form 
\begin{equation*}
\lambda^{k}(n) =  \sum_{j=1}^{n} \log \frac{f_1^k(X_j^k)}{f_0^k(X_j^k)}, \quad n \in \bN,
\end{equation*}
and it is a   random walk with  drift  $D_1^k$ under $\Pro_1^k$ and $-D_0^k$ under $\Pro_0^k$.  

Our goal  in this section is to show that the proposed multiple testing procedures in Section \ref{procedures}  are asymptotically optimal.  Our strategy for proving this is first to establish a \textit{non-asymptotic} lower bound on the minimum possible expected sample size  in $\Delta_{\alpha,\beta}(\cP)$ for some arbitrary class $\cP$, and then show that this lower bound is  attained  by the gap rule when  $\cP=\cP_m$ and by the gap-intersection rule when $\cP=\cP_{\ell,u}$  as  $\alpha, \beta \rightarrow  0$.

 
\subsection{A lower bound on the optimal performance}
In order to state the lower bound on the optimal performance, we introduce the  function
 \begin{equation} \label{phi}
 \varphi(x,y) := x\log\left(\frac{x}{1-y} \right) + (1-x) \log\left(\frac{1-x}{y} \right), \quad x, y \in (0,1),
 \end{equation}
and for any  subsets  $\MC , \cA \subset [K]$ such that  $\MC \neq \cA$ we set
\begin{equation*} 
\gamma_{\cA,\MC}(\alpha,\beta) := \begin{cases}
\varphi(\alpha,\beta), \quad &\text{ if } \MC\setminus \cA \neq \emptyset , \; \cA\setminus \MC = \emptyset, \\
\varphi(\beta,\alpha),   &\text{ if } \MC\setminus \cA = \emptyset,\; \cA\setminus \MC \neq \emptyset, \\
\varphi(\alpha,\beta) \vee \varphi(\beta,\alpha), \quad &\text{ otherwise}.
\end{cases}
\end{equation*}

\begin{theorem} \label{lower_bound}
For any class $\cP$,  $\cA \in \cP$ and $\alpha, \beta \in (0,1)$ such that  $\alpha + \beta < 1$ we have
\begin{align}\label{lower_bound_formula}
\inf_{(T,d) \in \Delta_{\alpha,\beta}(\cP)} \Exp_\cA [T]
 \geq 
\max_{\MC \in \cP, \MC \not = \cA} \, 
\frac{\gamma_{\cA,\MC}(\alpha,\beta)}
{\sum_{k \in \cA\setminus \MC} D_1^{k} +
\sum_{k \in \MC\setminus \cA} D_0^k} .
\end{align}
\end{theorem}

\begin{proof} 
Fix  $(T,d) \in \Delta_{\alpha,\beta}(\cP)$ and $\cA \in \cP$. 
Without loss of generality, we  assume that $\Exp_\cA [T] < \infty$. 
For any  $\MC \in \cP$ such that $\MC\neq \cA$,  the  log-likelihood ratio process   $\lambda^{\cA,\MC}$, defined in \eqref{ll_diff},  is a  random walk under $\Pro_{\cA}$  with drift  equal to  
$$
\Exp_{\cA}[ \lambda^{\cA,\MC}(1)] = \sum_{k \in \cA\setminus \MC} D_1^{k} + \sum_{k \in \MC\setminus \cA} D_0^{k},
$$
since each  $\lambda^{k}$ is a   random walk with  drift  $D_1^k$ under $\Pro_1^k$ and $-D_0^k$ under $\Pro_0^k$. Thus,  from  Wald's identity  it follows that
\begin{equation*}  
\Exp_\cA [T]  = \frac{\Exp_\cA [\lambda^{\cA,\MC}(T) ]}
{\sum_{k \in \cA\setminus \MC} D_1^{k} + \sum_{k \in \MC\setminus \cA} D_0^k},
\end{equation*}
and  it suffices to show that  for any $\MC \in \cP$ such that $\MC \neq \cA$ 
we have 
\begin{equation} \label{LBB}
\Exp_\cA [\lambda^{\cA,\MC} (T) ] \geq \gamma_{\cA, \MC} (\alpha,\beta). \end{equation} 
Suppose that $\MC\setminus \cA \neq \emptyset$ and let  $j \in \MC\setminus \cA$. Then, from Lemma~\ref{KL_bound}  in the Appendix  we have
\begin{align*}
\Exp_\cA \left[\lambda^{\cA,\MC}(T) \right] &= \Exp_\cA \left[\log  \frac{d\Pro_\cA}{d\Pro_\MC} (\MF_T)\right] 
\geq \varphi \left(\Pro_\cA(d^{j} = 1), \Pro_\MC(d^{j} = 0)  \right).
\end{align*}
By the definition of  $\Delta_{\alpha,\beta}(\cP)$,  we have 
$\Pro_\cA(d^{j} = 1) \leq \alpha$ and   $\Pro_{\MC}(d^{j} = 0) \leq \beta$. Since the function $\varphi(x,y)$ is decreasing on the set $\{(x,y): x+y \leq 1\}$, and by assumption $\alpha + \beta \leq 1$, we conclude that if $\MC \setminus \cA \neq \emptyset$, then
$$
\Exp_\cA [\lambda^{\cA,\MC}(T) ] \geq \varphi(\alpha,\beta).
$$
With a symmetric argument  we can show that if $\cA  \setminus \MC \neq \emptyset$, then
$$
\Exp_\cA [\lambda^{\cA,\MC}(T) ] \geq \varphi(\beta,\alpha).
$$
 The two last inequalities imply \eqref{LBB}, and this  completes the proof. 
\end{proof}
\begin{remark}\label{remark_on_phi}
By the definition of $\varphi$ in~\eqref{phi}, we have
\begin{equation}\label{phi_lemma}
\varphi(\alpha,\beta) = |\log \beta|\,(1 + o(1)),\quad 
\varphi(\beta,\alpha) = |\log \alpha \,|(1 + o(1))
\end{equation}
 as $\alpha,\beta \to 0$ at arbitrary rates. 
\end{remark}


\subsection{Asymptotic optimality of the proposed schemes}

In what follows, we assume that for each stream $k \in [K]$ we have:
 \begin{equation} \label{second}
  \int \left( \log \left(\frac{f_0^k}{f_1^k} \right) \right)^2 f_i^k d\mu^k < \infty, \quad  i=0,1.
 \end{equation}
Although this assumption is not necessary for the asymptotic optimality of the proposed rules to hold,  it will allow us to use Lemma \ref{aux} in the Appendix and  obtain valuable insights regarding the effect of  
prior information on the optimal  performance. Moreover, for each subset $\cA\subset [K]$ we set:
\begin{equation*}
\eta_{1}^{\cA} := \min_{k \in \cA} D_1^k ,\qquad \eta_{0}^{\cA} := \min_{j \notin \cA} D_0^j,
\end{equation*}
and, following  the convention that the minimum over the empty set is $\infty$, we define:
 $\eta_1^{\emptyset}= \eta_0^{[K]}:= \infty$.

\subsubsection{Known number of signals} \label{sec: known_size}
We will first show that the gap rule, defined in \eqref{gap}, is asymptotically optimal with respect to class $\cP_m$, where  $1 \leq m \leq K-1$.  In order to do so,  we start  with  an  upper bound on the  expected sample size of this procedure.  
\begin{lemma}\label{ESS_TG}
Suppose that assumption ~\eqref{second} holds. Then, for any $\cA \in \cP_{m}$, as $c \to \infty$ we have
\begin{equation*} 
\Exp_\cA[T_G]  \leq \frac{c}{  \eta_1^{\cA} +
 \eta_0^{\cA} }  + O\left( m (K-m) \sqrt{c} \right).
\end{equation*}
\end{lemma}

\begin{proof}
Fix  $\cA \in \cP_{m}$. For any $c>0$  we have  $T_G \leq T_G'$, where $T_G'$ is defined in~\eqref{T_G_prime}, and it is the first time that all $m(K-m)$ processes  of the form $\lambda^{k} - \lambda^{j}$ with $k \in \cA \; \text{and} \; j \notin \cA$ exceed  $c$. 
Due to condition~\eqref{second},  each   $\lambda^{k} - \lambda^{j}$  with $k \in \cA \; \text{and} \; j \notin \cA$ is a random walk under $\Pro_\cA$ 
 with positive drift $D_1^{k} + D_0^{j}$ and finite second  moment. Therefore, from Lemma~\ref{aux} 
 it follows that   as $c \to \infty$:
\begin{align*}
\Exp_\cA [T_G']  \leq c \left( \min_{k \in \cA, j \notin \cA} (D_1^{k} + D_0^{j})\right)^{-1}
 + O\left(  m(K-m) \sqrt{c} \right),
\end{align*}
and this completes the proof, since 
$\min_{k \in \cA, j \notin \cA} (D_1^{k} + D_0^{j}) =
 \eta_1^{\cA} + \eta_0^{\cA}
$.
\end{proof}


The next theorem establishes the asymptotic optimality of the gap rule.

\begin{theorem}\label{asymptotic_optimality_TG}
Suppose assumption~\eqref{second} holds and 
let the threshold $c$ in the gap rule be selected according to~\eqref{thres}.   Then for every $\cA \in \cP_m$, we have as $\alpha,\beta \to 0$
\begin{equation*}
\Exp_\cA[T_G] \;\sim \; 
\frac{|\log(\alpha\wedge \beta)|}{\eta_1^{\cA} +
 \eta_0^{\cA}} \;  \sim\; 
\inf_{(T,d) \in \Delta_{\alpha,  \beta}(\cP_m)} \Exp_\cA [T].
\end{equation*}
\end{theorem}

\begin{proof} 
Fix $\cA \in \cP_m$. If thresholds are selected according to ~\eqref{thres}, then from Lemma~\ref{ESS_TG}  it follows that as $\alpha,\beta \to 0$ 
\begin{equation}\label{upper_bound_TG2}
\Exp_\cA[T_G]  \leq   \frac{|\log(\alpha\wedge \beta)| } { \eta_1^{\cA} + \eta_0^{\cA} } + O\left( m (K-m) \sqrt{|\log (\alpha\wedge \beta) |} \right).
\end{equation}
Therefore, it suffices to show that  the lower bound in Theorem~\ref{lower_bound} agrees  with the upper bound in~\eqref{upper_bound_TG2}   in the first-order term   as $\alpha,\beta\to 0$.  To see this, note that for any $\MC \in \cP_m$ such that $\MC \neq  \cA$  we have   $\MC \setminus \cA
\neq \emptyset$ \textit{and}  $\cA \setminus \MC \neq \emptyset$, and consequently 
 $$\gamma_{\cA,\MC}(\alpha,\beta)=  \varphi(\alpha, \beta)  \vee \varphi(\beta, \alpha).
 $$
  This means that the numerator  in~\eqref{lower_bound_formula} does not depend on $\MC$.  Moreover, if we restrict our attention to subsets in $\cP_m$ that differ from $\cA$ in two streams, i.e., subsets of the form $\MC = \cA \cup \{j\} \setminus \{k\} $ for  some $k \in \cA$ and $j \notin \cA$, for which 
$$\sum_{i \in \cA\setminus \MC} D_1^{i} + \sum_{i \in \MC\setminus \cA} D_0^i  =  D_1^{k} + D_0^{j},
$$
then we have
\begin{align*}  
\min_{\MC \in \cP_m, \MC \neq    \cA  }  \left[
 \sum_{i \in \cA\setminus \MC} D_1^{i} + \sum_{i \in \MC\setminus \cA} D_0^i \right] &\leq \min_{k \in \cA , j\notin \cA}  \, \left[D_1^{k} + D_0^{j} \right]  = \eta_1^{\cA} + \eta_0^{\cA}.
\end{align*}
By the last inequality and  Theorem~\ref{lower_bound}  we obtain the following non-asymptotic lower bound, which holds for  any $\alpha, \beta$ such that $\alpha+\beta<1$:
\begin{align*} 
\inf_{ (T,d) \in \Delta_{\alpha,\beta}(\cP_m)} \Exp_\cA [T ] 
\geq  \frac{\max\{ \varphi(\alpha, \beta) , \varphi(\beta, \alpha)\} } { \eta_1^{\cA} + \eta_0^{\cA} } .
\end{align*}
By~\eqref{phi_lemma}, we have as $\alpha, \beta \rightarrow 0$ 
$$
\max\{ \varphi(\alpha, \beta) , \varphi(\beta, \alpha)\} =  |\log (\alpha \wedge \beta)|\,(1 + o(1)).
$$
Consequently,
\begin{align*} 
\inf_{(T,d) \in \Delta_{\alpha,\beta}(\cP_m)} \Exp_\cA \left(T\right) 
\geq   \frac{|\log (\alpha\wedge \beta)|} { \eta_1^{\cA} +
 \eta_0^{\cA}}\,(1 + o(1)),
\end{align*}
which  completes the proof. 
\end{proof}

\begin{remark} \label{r1}
It is interesting to consider the special case of identical hypotheses, in which $f_1^k=f_1$ and $f_0^k=f_0$, and consequently  $D_1^k=D_1$ and $D_0^k=D_0$ for every $k \in [K]$. Then, 
 $\eta_1^{\cA}=D_1$ and  $\eta_0^{\cA}=D_0$  for every $\cA \subset [K]$, and  from Theorem~\ref{asymptotic_optimality_TG} it follows that  the  \textit{first-order} asymptotic approximation  to the expected sample size of the gap rule (as well as to the  optimal expected sample size within 
$\Delta_{\alpha,\beta}(\cP_m)$),  $|\log(\alpha\wedge \beta)| / ( D_1+D_0)$,   is  independent of the  number of signals, $m$. We should  stress that  this does not mean that the \textit{actual}  performance of the gap rule  is independent of $m$. Indeed,  the  second  term  in the right-hand side of~\eqref{upper_bound_TG2} suggests that  the smaller $m(K-m)$ is, i.e., the further away the proportion of signals $m/K$ is from $1/2$,  the smaller the expected sample size of the gap rule will be. This intuition will be corroborated by the simulation study in Section \ref{simulation} (see Fig.~\ref{fig:three}). 
\end{remark}



\subsubsection{Lower and upper bounds on the number of  signals}

We will now show that the gap-intersection rule, defined in \eqref{gap-intersection}, is asymptotically optimal 
with respect to class $\cP_{\ell,u}$ for some  $0\leq \ell < u \leq K$.  As before, we start with establishing an upper bound on the expected sample size of this rule.
\begin{lemma}  \label{ESS_LU}
Suppose that assumption~\eqref{second} holds. Then,  for any $\cA \in \cP_{\ell,u}$ we have  as $a,b,c,d \to \infty$
\begin{equation*}
\Exp_\cA [T_{GI} ] \leq
\begin{cases}
\max\left\{ a/ \eta_0^{\cA}  \; , \; c / (\eta_0^{\cA} + \eta_1^{\cA}) \right\} (1+o(1)) \; &\text{ if }\; |\cA| = \ell\\
\max\left\{ a/ \eta_0^{\cA} \; ,\;  b/ \eta_1^{\cA}\right\} 
+ O(K  \sqrt{a \vee b})
\; &\text{ if }\; \ell < |\cA| < u\\
\max\left\{  b/ \eta_1^{\cA} \; ,\;  d / (\eta_0^{\cA} + \eta_1^{\cA})  \right\}  (1+o(1)) \; &\text{ if }\; |\cA| = u
\end{cases} 
\end{equation*} 
Furthermore, if $c - a = O(1)$ and $d-b = O(1)$, then
\begin{equation} \label{GIsecondorder}
\Exp_\cA [T_{GI} ] \leq 
\begin{cases}
a/ \eta_0^{\cA}  +  O((K- \ell)  \sqrt{a} )\quad &\text{ if }\; |\cA| = \ell\\
b/ \eta_1^{\cA} +O(u  \sqrt{b} ) \quad &\text{ if }\; |\cA| = u
\end{cases} 
\end{equation} 
\end{lemma}

\begin{proof}
Fix  $\cA \in \cP_{\ell,u}$.  By the definition of the stopping time $T_{GI}$, $$\Exp_\cA [T_{GI}] \leq \min \left\{\Exp_\cA[\tau_1], \Exp_\cA[\tau_2],\Exp_\cA[\tau_3] \right\}.$$ 

Suppose first   $\ell < |\cA| <u$ and observe that $\tau_2 \leq \tau_2'$, where 
$\tau_2'$ is defined in~\eqref{tau2_prime}. Under condition~\eqref{second},
for every $k \in \cA$ and $j \notin \cA$, 
  $-\lambda^{j}$  and $\lambda^{k}$ are random walks with finite second moments and positive drifts $D_0^{j}$ and $D_1^{k}$, respectively. Therefore, from Lemma \ref{aux} we have that 
$$
\Exp_{\cA}[\tau'_2]  \leq \max\left\{ a/ \eta_0^{\cA} \; ,\;  b/ \eta_1^{\cA}\right\}  +   O ( K \sqrt{a \vee b}  ) .
$$

Suppose now that $|\cA| = \ell$ and observe that $\tau_1 \leq \tau_1'$, where 
$$
\tau_1' := \inf\{n \geq 1: -\lambda^{j}(n) \geq a,\; \lambda^{k}(n) - \lambda^{j}(n) \geq c \; \; \text{for every} \;  k \in \cA, j \notin \cA \},
$$
where $-\lambda^{j}$ and $\lambda^{k}- \lambda^{j}$ are random walks with
finite second moments and  positive drifts $D_0^{j}$ and $D_1^{k} + D_0^{j}$, respectively. The result  follows again  from an application of  Lemma \ref{aux}. 
If in addition we have that $c - a = O(1)$, then  $\tau_1 \leq \tau_1''$, where
$$
\tau_1'' := \inf\{n \geq 1: -\lambda^{j}(n) \geq a,\; \lambda^{k}(n) \geq c - a  \; \; \text{for every} \;  k \in \cA, j \notin \cA \}.
$$
Therefore, the second part of the lemma  follows again  from an application of Lemma \ref{aux}. 
\end{proof}



The next theorem establishes the asymptotic optimality of the gap-inter\-section rule. 

\begin{theorem}\label{asymptotic_optimality_MI}
Suppose that assumption~\eqref{second} holds  and let the thresholds in the gap-inter\-section rule be selected according to~\eqref{thres_LU}. Then
for any $\cA \in \cP_{\ell,u}$, we have as $\alpha,\beta \to 0$  
\begin{align*}
\Exp_\cA [T_{GI} ] \;&\sim \;
\inf_{(T,d) \in \Delta_{\alpha,\beta} (\cP_{\ell,u}) } \Exp_\cA [ T ] \\
&\sim\;  
\begin{cases}
\max\left\{|\log \beta| /   \eta_0^{\cA} \; , \; |\log \alpha|/ (\eta_0^{\cA} + \eta_1^{\cA}) \right\} \quad &\text{ if }\; |\cA| = \ell\\
\max\left\{ |\log \beta| /   \eta_0^{\cA} \; ,\; |\log \alpha| / \eta_1^{\cA} \right\} \quad &\text{ if }\; \ell <|\cA| < u \\
\max\left\{ |\log \alpha| / \eta_1^{\cA} \; , \; |\log \beta| / (\eta_0^{\cA} + \eta_1^{\cA}) \right\} \quad &\text{ if }\; |\cA| = u
\end{cases} .
\end{align*}
\end{theorem}

\begin{proof} 
Fix $\cA \in \cP_{\ell,u}$.  We will prove the result only in the case that  $|\cA| = \ell$, as the other two cases can be proved similarly.  If thresholds are selected according to \eqref{thres_LU}, then from  Lemma~\ref{ESS_LU} it follows that 
$$
\Exp_\cA [T_{GI} ] \leq  \max\left\{ \frac{|\log \beta|}{\eta_0^{\cA}} \; , \; 
\frac{|\log \alpha|}{ \eta_0^{\cA} + \eta_1^{\cA} } \right\}  \, (1+o(1))  .
$$

 Thus, it suffices to show that this asymptotic upper bound agrees asymptotically, \textit{up to a first order},  with the lower bound in Theorem~\ref{lower_bound}. Indeed, if $\MC$ is a subset in $\cP_{\ell,u}$ that has one more stream than $\cA$, i.e., $\MC = \cA \cup \{j\}$ for some  $j \notin \cA$, then
$$
\frac{\gamma_{\cA,\MC}(\alpha,\beta)}
{\sum_{i \in \cA\setminus \MC} D_1^{i} +
\sum_{i \in \MC\setminus \cA} D_0^i} =  \frac{  \varphi(\alpha, \beta) }{ D_0^{j}}.
$$
Further, consider $\MC = \cA \cup \{j\}/\{k\} \in \cP_{\ell,u}$ for some $k \in \cA$ and $j \notin \cA$, then 
$$
\frac{\gamma_{\cA,\MC}(\alpha,\beta)}
{\sum_{i \in \cA\setminus \MC} D_1^{i} +
\sum_{i \in \MC\setminus \cA} D_0^i} 
=  \frac{  \max\{ \varphi(\alpha, \beta) , \varphi(\beta, \alpha)\}}{ D_1^{k} + D_0^{j}}.
$$
Therefore,  from~\eqref{lower_bound}  it follows that
for every $\alpha, \beta$ such that $\alpha+\beta <1$
\begin{align*} 
\inf_{(T,d) \in \Delta_{\alpha,\beta}(\cP_{\ell,u})} \Exp_\cA [T ]
&\geq \max_{k \in \cA, j \notin \cA} \max\left\{  \frac{  \varphi(\alpha, \beta) }{ D_0^{j}} \; , \;  \frac{  \max\{ \varphi(\alpha, \beta) , \varphi(\beta, \alpha)\}}{ D_1^{k} + D_0^{j}}   \right\} \\
&= \max\left\{\frac{ \varphi(\alpha, \beta)} { \eta_0^{\cA}} ,\; \frac{ \varphi(\beta, \alpha)}{ \eta_1^\cA + \eta_0^{\cA} } \right\}.
\end{align*}
From~\eqref{phi_lemma}  it follows  that as  $\alpha,\beta\to 0$ 
$$
\inf_{(T,d) \in \Delta_{\alpha,\beta}(\cP_{l,u})} \Exp_\cA [T ]
\geq 
\max\left\{\frac{ |\log \beta|} { \eta_0^{\cA}} ,\; \frac{|\log \alpha|}
{ \eta_1^\cA + \eta_0^{\cA} } \right\}\, (1 +o(1)),
$$
which completes the proof. 
\end{proof}


\subsection{The case of no prior information} \label{noprior}
Recall that when we set $\ell=0$ and $u=K$, the \luName rule reduces to the intersection rule, defined in \eqref{intersection}. Therefore,  setting  $\ell=0$ and $u=K$ in Theorem \ref{asymptotic_optimality_MI}  we immediately obtain  that the intersection rule  is asymptotically optimal  in the  case of no prior information, i.e., with respect to class $\cP_{0,K}$;
this is itself a  new result to the best of our knowledge. However,  a more surprising corollary  of Theorem \ref{asymptotic_optimality_MI} is that 
the intersection rule, which does not use any prior information, is  asymptotically optimal even if bounds on the  number of signals are available, when  the following conditions are satisfied:
\begin{enumerate}
\item[(i)] the error probabilities are of the same order of magnitude,  in the sense that  $|\log \alpha| \sim  |\log \beta|$, 
\item[(ii)]  the hypotheses are identical and symmetric, in the sense that $D_1^k= D_0^k= D$ for every $k \in [K]$. 
\end{enumerate}

 On the other hand, a comparison with Theorem
 \ref{asymptotic_optimality_TG} reveals that, even in this special case,  
 the intersection rule  is never asymptotically optimal when the exact umber of signals is known in advance, in which case it  requires roughly \textit{twice} as many observations on average as the gap rule for the same precision level. 
 The following corollary summarizes these observations.

\begin{corollary}\label{intersection_optimiality}
Suppose that assumption \eqref{second} holds and that  the thresholds in the intersection rule are  selected according to \eqref{thres-intersection}. Then,  
for any $\cA \subset [K]$ we have  as $\alpha, \beta \rightarrow 0$ 
\begin{equation} \label{I_secondorder}
\Exp_\cA [T_{I} ] \leq  \max \left\{ 
\frac{|\log \alpha|}{ \eta_1^{\cA}}
, \frac{|\log \beta|}{ \eta_0^{\cA}} \right\}
+ O(K   \sqrt{|\log (\alpha \wedge \beta)|} ).
\end{equation} 
Further, the intersection rule is asymptotically optimal in the class $\Delta_{\alpha,\beta}(\cP_{0,K})$, i.e.,
as $\alpha,\beta \to 0$ 
\begin{align*}
\Exp_\cA [T_{I} ]  \; \sim\;    \max \left\{ 
\frac{|\log \alpha|}{ \eta_1^{\cA}}
, \frac{|\log \beta|}{ \eta_0^{\cA}} \right\}  \; \sim \;
\inf_{(T,d)  \in \Delta_{\alpha,\beta}(\cP_{0,K})} \Exp_\cA [T].
\end{align*}
In the special case that   $|\log \alpha| \sim  |\log \beta|$  and  $D_1^{k}=D_0^{k}=D$ for every $k \in [K]$, 
 \begin{align*}
\Exp_\cA [T_{I} ]
&\sim\;   \frac{|\log \alpha|}{ D } \; \sim \;
\inf_{(T,d)  \in \Delta_{\alpha,\beta}(\cP_{\ell,u})} \Exp_\cA [ T ] \quad  \text{for every} \; \cA \in \cP_{\ell, u}, \\
\Exp_\cA [T_{I} ]
& \sim\;   \frac{|\log \alpha|}{ D } \; \sim \;  
2 \,  \inf_{(T,d)  \in \Delta_{\alpha,\beta}(\cP_{m})} \Exp_\cA [T]  \quad \text{for every} \; \cA \in \cP_{m},
\end{align*}
for every $0 \leq \ell <u\leq  K$ and $1 \leq  m \leq K-1$.

\end{corollary}

\begin{remark}
Corollary \ref{intersection_optimiality}  implies that, in the special symmetric  case that $|\log \alpha| \sim  |\log \beta|$  and  $D_1^{k}=D_0^{k}=D$,  prior lower and upper bounds on the true number of signals do not improve the optimal expected sample size  up to a  \textit{first-order} asymptotic approximation. However,  a comparison between the second-order terms in~\eqref{GIsecondorder} and~\eqref{I_secondorder} suggests that  such prior  information does improve the optimal performance, an intuition that will be corroborated by the simulation study in Section \ref{simulation} (see Fig.~\ref{fig:three}). 
\end{remark}


\begin{remark}
In addition to the intersection rule,  \citet{de2012sequential} proposed the ``incomplete rule'', $(T_{\max},d_{\max})$, which is defined as 
$$T_{\max} := \max\{\sigma_1,\ldots, \sigma_K\}\; \text{ and } \; 
d_{\max}:= (d_{\max}^1, \ldots ,d_{\max}^K),$$ 
where for every  $k \in [K]$ we have
\begin{align} \label{incomplete_rule}
\sigma_k &:= \inf \left\{n \geq 1: \lambda^k(n) \not \in (-a,b) \right\}
, \quad 
d_{\max}^k := 
\begin{cases}
1,\quad \text{ if } \lambda^k(\sigma_k) \geq b \\
0, \quad\text{ if } \lambda^k(\sigma_k) \leq -a
\end{cases}.
\end{align}
According to this rule,  each stream is sampled until the corresponding test statistic exits the interval $(-a,b)$, \textit{independently of the other streams}. It is clear that, for the same thresholds $a$ and $b$,  $T_{\max} \leq T_{I}$. Moreover, with  a direct application of Boole's inequality, as  in~\citet{de2012sequential}, it follows that selecting the  thresholds according to  \eqref{thres-intersection}   guarantees  the desired error control for the incomplete rule. Therefore,   Corollary~\ref{intersection_optimiality} remains valid if we replace the intersection rule with the incomplete rule.
\end{remark}

\section{Simulation study} \label{simulation}

\subsection{Description}

In this section we present a simulation study whose  goal is to corroborate the asymptotic results and insights of   Section~\ref{theory} in the symmetric case described in Corollary~\ref{intersection_optimiality}. Thus, we set $K = 10$ and  let $ f_i^k =\mathcal{N}(\theta_i,1)$ for each $k \in [K]$, $i = 0,1$, where $\theta_0 =0, \theta_1=0.5$, in which case $D_0^k = D_1^k =D= (1/2) (\theta_1)^2= 1/8$, and the distribution of $\lambda^k$ under $H_1^k$ is the same as $-\lambda^k$ under $H_0^k$. Furthermore, we set $\alpha = \beta$. This is a convenient setup for simulation purposes, since the expected sample size and the two familywise errors of each proposed procedure are the same for all scenarios with the same number of signals, i.e. for all $\cA$'s with the same size.



For any user specified level $\alpha$, we have two ways to determine the critical value of each procedure. First, we can use upper bound on the error probability to compute conservative threshold (\eqref{thres} for the gap rule, and~\eqref{thres_LU} for the gap-intersection rule). Second, we can apply  the importance sampling technique of Section~\ref{is} to determine non-conservative threshold, such that the \emph{maximal} familywise type I error probability is controlled \emph{exactly} at level $\alpha$. As we see in Fig.~\ref{fig:std_over_est}, the relative errors of  the proposed  Monte Carlo estimators, even  for error probabilities of the order $ 10^{-8}$, are smaller than $1.5\%$ for the  gap rule, $8\%$ for the gap-intersection rule, $1\%$ for the intersection rule.


\begin{figure}
\subfloat[Gap rule]{
\includegraphics[width=0.31\linewidth]{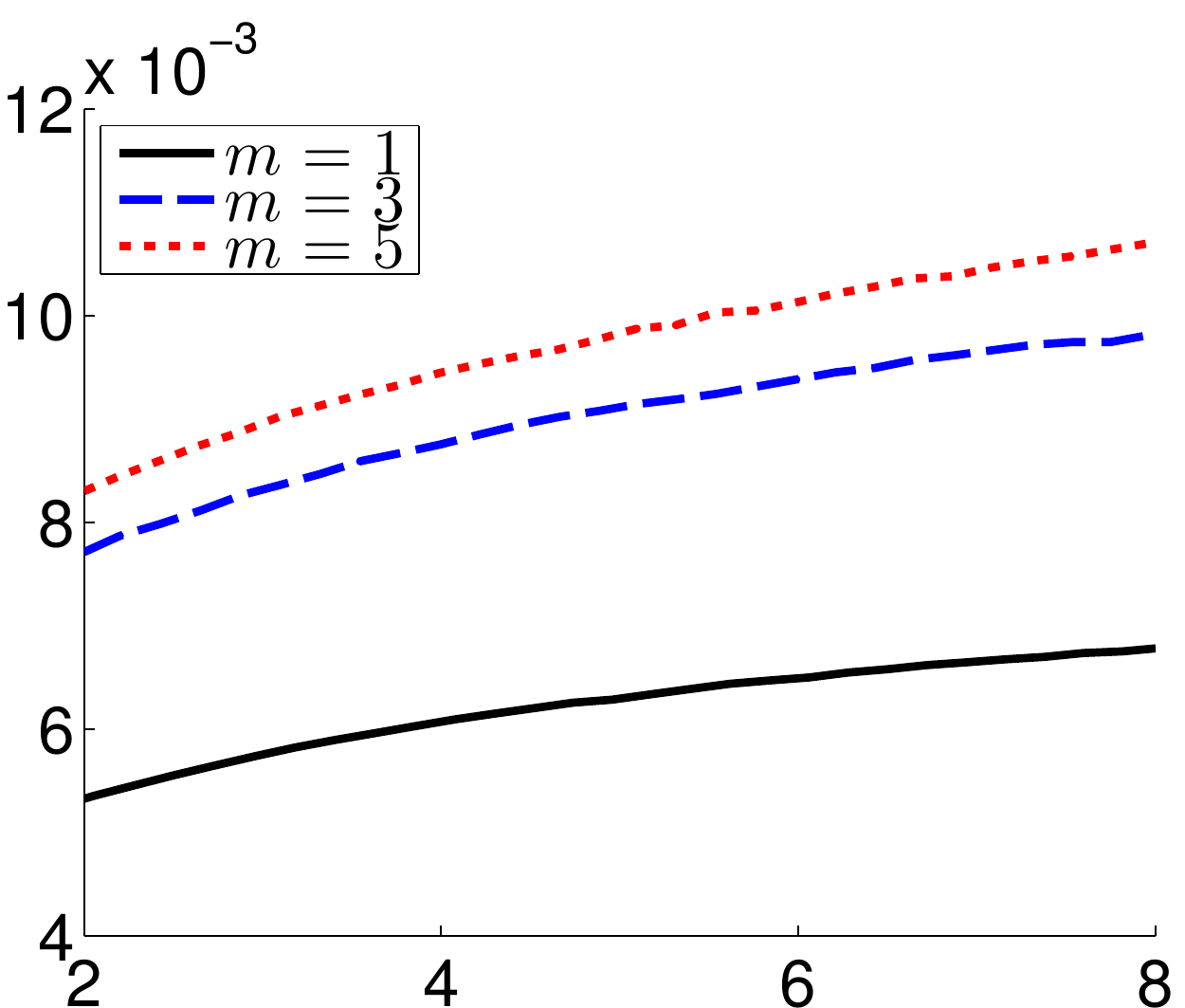}  
}
\subfloat[Gap-intersection rule]{
\includegraphics[width=0.31\linewidth]{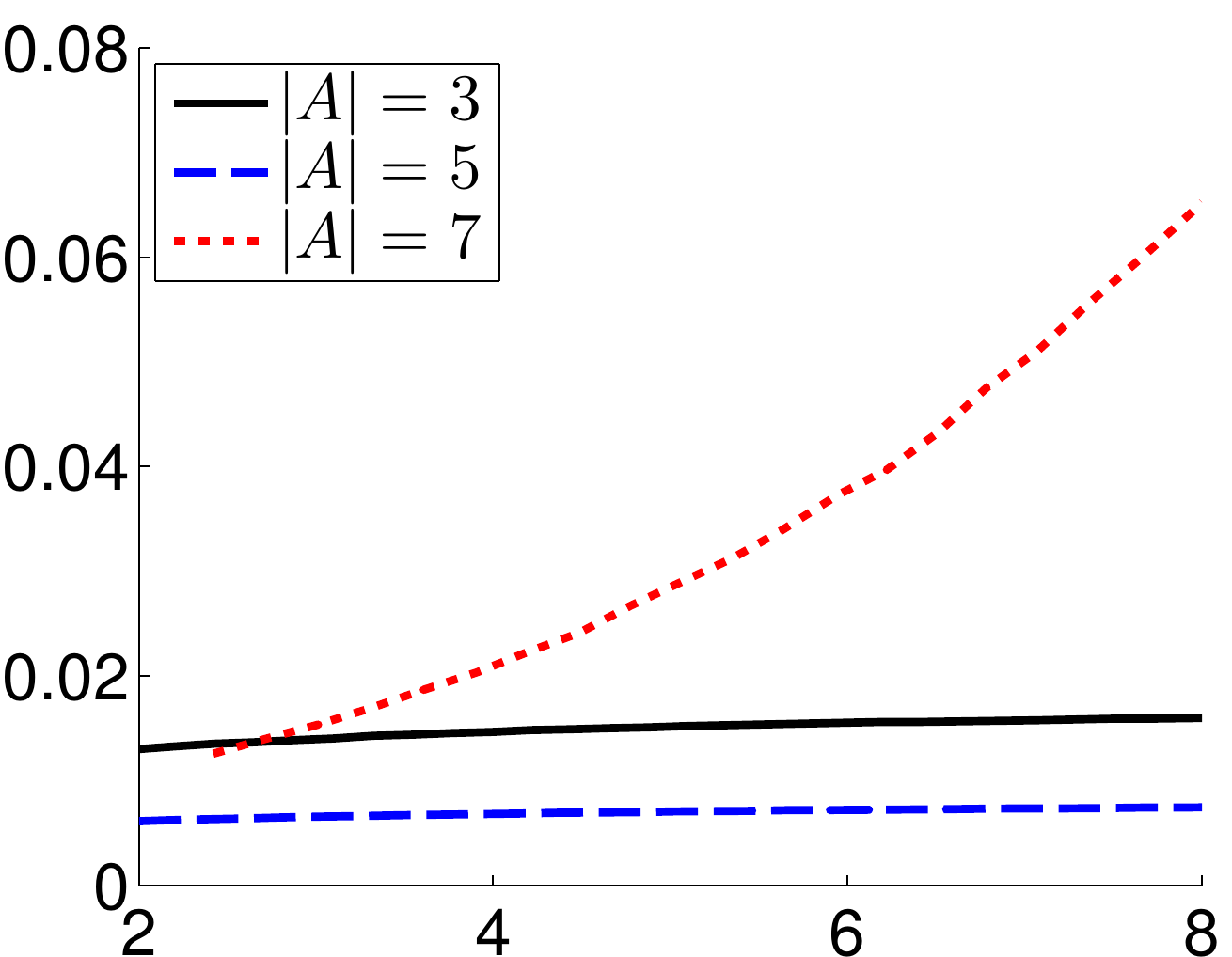}  
}
\subfloat[Intersection rule]{
\includegraphics[width=0.31\linewidth]{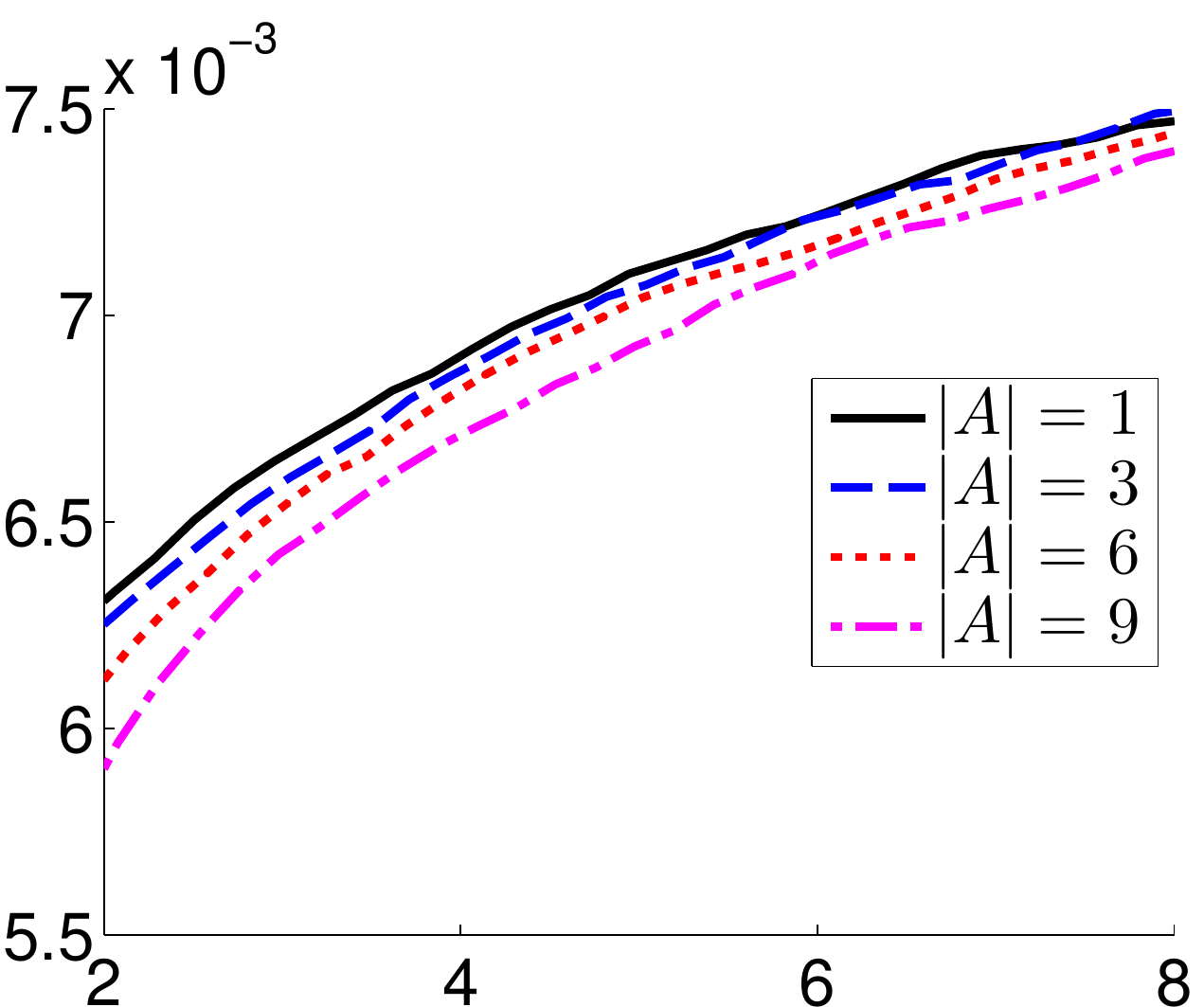}  
}
\caption{The x-axis is  $|\log_{10}(\Pro_{\cA}(\cA \lesssim d))|$. The y-axis is the relative error of the estimate of the familywise type-I error, $\Pro_{\cA}(\cA \lesssim d)$, that is the  ratio of the standard deviation of the estimate  over the estimate itself. Each curve is computed based on $100,000$ realizations.
}
\label{fig:std_over_est}
\end{figure}

\subsubsection{Gap rule}
First, we consider the case  in which the number of signals is known  to be equal to $m$ $(\cP = \cP_m)$ for  $m \in \{1, \ldots, 9\}$, and  we can apply  the corresponding  gap rule, defined in \eqref{gap}. Due to the symmetry of our setup, the expected sample size $\Exp_{\cA}[T_{G}]$ and the error probability $\Pro_\cA(d_{G} \neq \cA)$ are the same for $\cA \in \cP_m$ and $\cA \in \cP_{K-m}$; thus, it suffices to consider $m$ in $\{1,\ldots,5\}$
, and an \textit{arbitrary} $\cA \in \cP_m$ for fixed $m$.

We start with non-conservative critical value determined by Monte Carlo method.  For each  $m \in \{1,3,5\}$  and some   $\cA \in \cP_m$,   we consider $\alpha$'s ranging from $10^{-2}$ to $10^{-8}$. For each such $\alpha$, we compute the  threshold $c$ in the gap-rule  that  guarantees  $\alpha=\Pro_\cA(d_{G} \neq \cA)$, and then the expected sample size $\Exp_{\cA}[T_{G}]$ that corresponds to this threshold.  In  Fig.~\ref{fig:m} we plot   $\Exp_{\cA}[T_{G}]$ against $|\log_{10} (\alpha)|$ when $m=1,3, 5$.  In Table~\ref{tab:gap} we present the actual numerical results for $c = 10$.  

In  Fig.~\ref{fig:m} we also plot  the first-order asymptotic approximation to the optimal expected sample size  obtained in Theorem~\ref{asymptotic_optimality_TG}, which in this particular symmetric case takes the form  $|\log \alpha|/(2D)=4|\log \alpha|$. From our asymptotic theory we know that  the ratio of  $\Exp_{\cA}[T_G]$ over this quantity  goes to 1 as $\alpha \rightarrow 0$, and this convergence is  illustrated in  Fig.~\ref{fig:m_n}. 
 
Further, in Fig.~\ref{fig:gap_upper} we present 
for the case $\cP = \cP_3$ the expected sample size of the gap rule when its threshold is given by the  explicit expression in~\eqref{thres}, and compare it with the corresponding  expected sample size that is obtained with the sharp threshold, which is computed via simulation. 




\subsubsection{Gap-intersection rule}
Second, we consider the case in which  the number of signals  is known to be  between 3 and 7 ($\cP = \cP_{\ell,u} = \cP_{3,7}$), and we can apply the \luName rule, defined in~\eqref{gap-intersection}. Due to the  symmetry  of the setup and Lemma~\ref{error_control_LU}, we set $a = b$ and $c = d = b+ \log(u)= b +\log(7)$. 

As before, we consider  $\alpha$'s ranging from $10^{-2}$ to $10^{-8}$. For each such $\alpha$,  we obtain the threshold $b$ such that  $\max_{\cA} \Pro_{\cA}(\cA \lesssim d_{GI}) =\alpha$, where the maximum is taken over $\mathcal{A} \in \cP_{\ell, u}$, and then compute the corresponding expected sample size  $\Exp_{\cA}[T_{GI}]$ for every  $\mathcal{A} \in \cP_{\ell, u}$.    In Fig.~\ref{fig:lu}  we plot $\Exp_{\cA}[T_{GI}]$ against    $|\log_{10}(\alpha)|$ for $|\cA|=3$ and $5$, since by symmetry $\Exp_{\cA}[T_{GI}]$ is the same for $|\cA| = k$ and $10-k$, and the results for $|\cA| = 4$ and $5$ were too close. 
This is also   evident from Table~\ref{tab:MI}, where we present the numerical results for $b= 10$.   
In the same graph  we also plot the first-order asymptotic approximation to the optimal performance obtained in Theorem~\ref{asymptotic_optimality_MI}, which in this case is $|\log \alpha|/D= 8|\log \alpha|$.  By  Theorem~\ref{asymptotic_optimality_MI}, we know that the ratio of $\Exp_{\cA}[T_{GI}]$ over $8|\log \alpha|$ goes to 1 as $\alpha \rightarrow 0$, which is corroborated in  Fig.~\ref{fig:lu_n}. 


\subsubsection{Intersection versus incomplete rule}
Finally, we consider  the case of no prior information ($\cP = \cP_{0,10}$), in which we compare the  intersection rule  with the incomplete rule. This is a special case of   the previous setup  with $\ell=0$ and $u=K$, but now 
the expected sample size (for both schemes)  is the same for every subset of signals $\cA$, which allows us to plot only one curve for each scheme  in Fig.~\ref{fig:no_prior} (\textit{non-conservative} critical value is used). In the same graph we also 
plot  the first-order approximation to the optimal performance,  $|\log \alpha|/D= 8|\log \alpha|$, whereas in Fig.~\ref{fig:no_prior_n}. we plot the corresponding normalized version.

Further, in Fig.~\ref{fig:int_upper}  we present the expected sample size of the intersection rule when its threshold is given by the  explicit expression in~~\eqref{thres_LU}, and compare it with the corresponding  expected sample size that is obtained with the sharp threshold, which is computed via simulation. 



\subsection{Results}
There are a number of conclusions that  can be drawn from the presented graphs. First of all,  from Fig.~\ref{fig:m} it follows that the gap rule performs the best  when there are exactly $m=1$ or $9$ signals, whereas  its performance is quite similar for $m = 3,4,5$. As we mentioned before, this can be explained by the fact that the second  term in the right-hand side in \eqref{upper_bound_TG2} grows with  $m(K-m)$. 

Second, from Fig.~\ref{fig:lu} we can see that  the  \luName rule performs better in the boundary cases that  there are exactly 3 or 7 signals than in the  case of  5 signals, which can be explained by the second order term in~\eqref{GIsecondorder}.

Third, from Fig.~\ref{fig:no_prior} we can see that  the intersection rule is always better than the incomplete rule, although  they share the same prior information. 
 

Fourth, from the graphs in the second column of  Fig.~\ref{fig:three}  we can see that all curves approach 1, as expected from our asymptotic results; however, the convergence is relatively slow. This is reasonable,  as we do not divide the expected sample sizes by the optimal performance in each case, but with a strict lower bound on it instead.  

Fifth, comparing Fig.~\ref{fig:m} with Fig.~\ref{fig:lu} and~\ref{fig:no_prior},  we verify that knowledge of the exact number of signals roughly halves the required expected sample size in comparison to the case that we only have a lower and an upper bound on the number of signals.

Finally, we see by Tables~\ref{tab:gap} and~\ref{tab:MI} that  the upper bounds~\eqref{bound} and~\eqref{bound2} on the error probabilities
are  very crude. Nevertheless, from Fig.~\ref{fig:gap_upper} and~\ref{fig:int_upper}, we observe that  using these  conservative thresholds in the  design of the proposed procedures leads to  bounded performance loss as the error probabilities go to 0 relative to the case of sharp thresholds, obtained via Monte Carlo simulation. This is expected, as the expected sample size scales with the logarithm of the error probabilities.

\begin{figure}
\subfloat[Gap rule: $\cP = \cP_m$, $m = 1,3,5$]{
\includegraphics[width=0.45\linewidth]{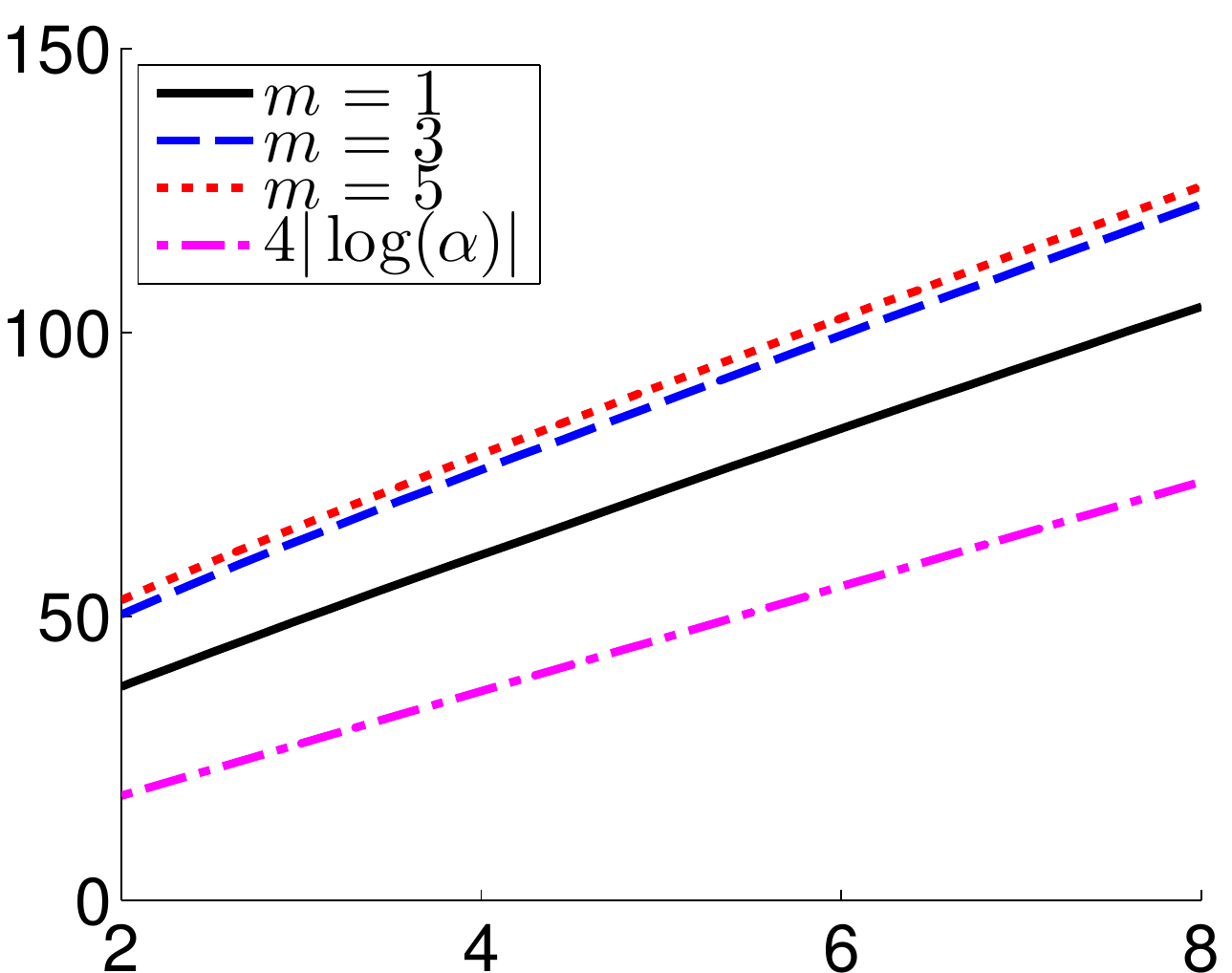}  
\label{fig:m}
}\hspace{0.2in}
\subfloat[Normalized by $4|\log(\alpha)|$]{
\includegraphics[width=0.44\linewidth]{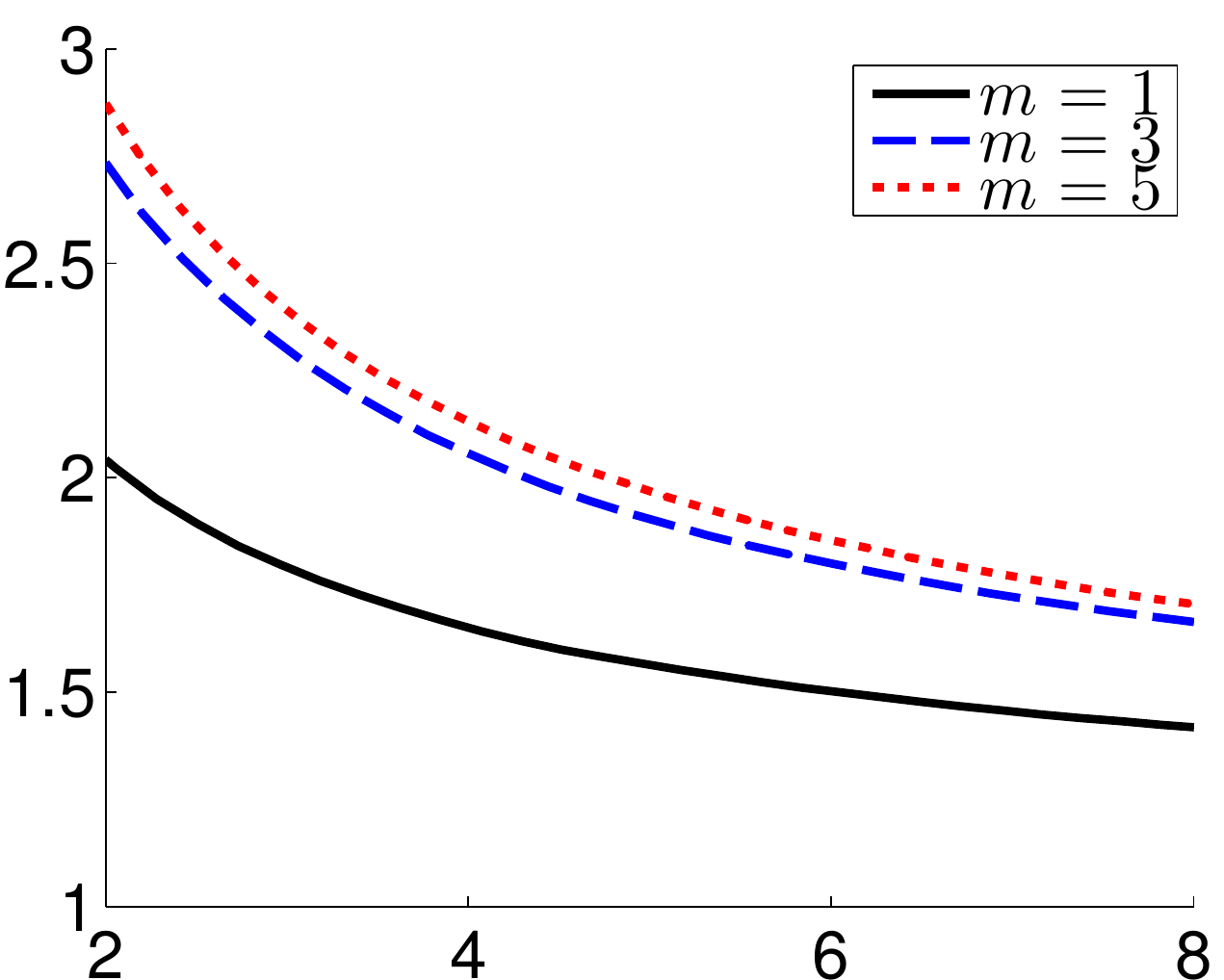}  
\label{fig:m_n}
} \\
\subfloat[Gap-intersection rule: $\cP = \cP_{3,7}$]{
\includegraphics[width=0.44\linewidth]{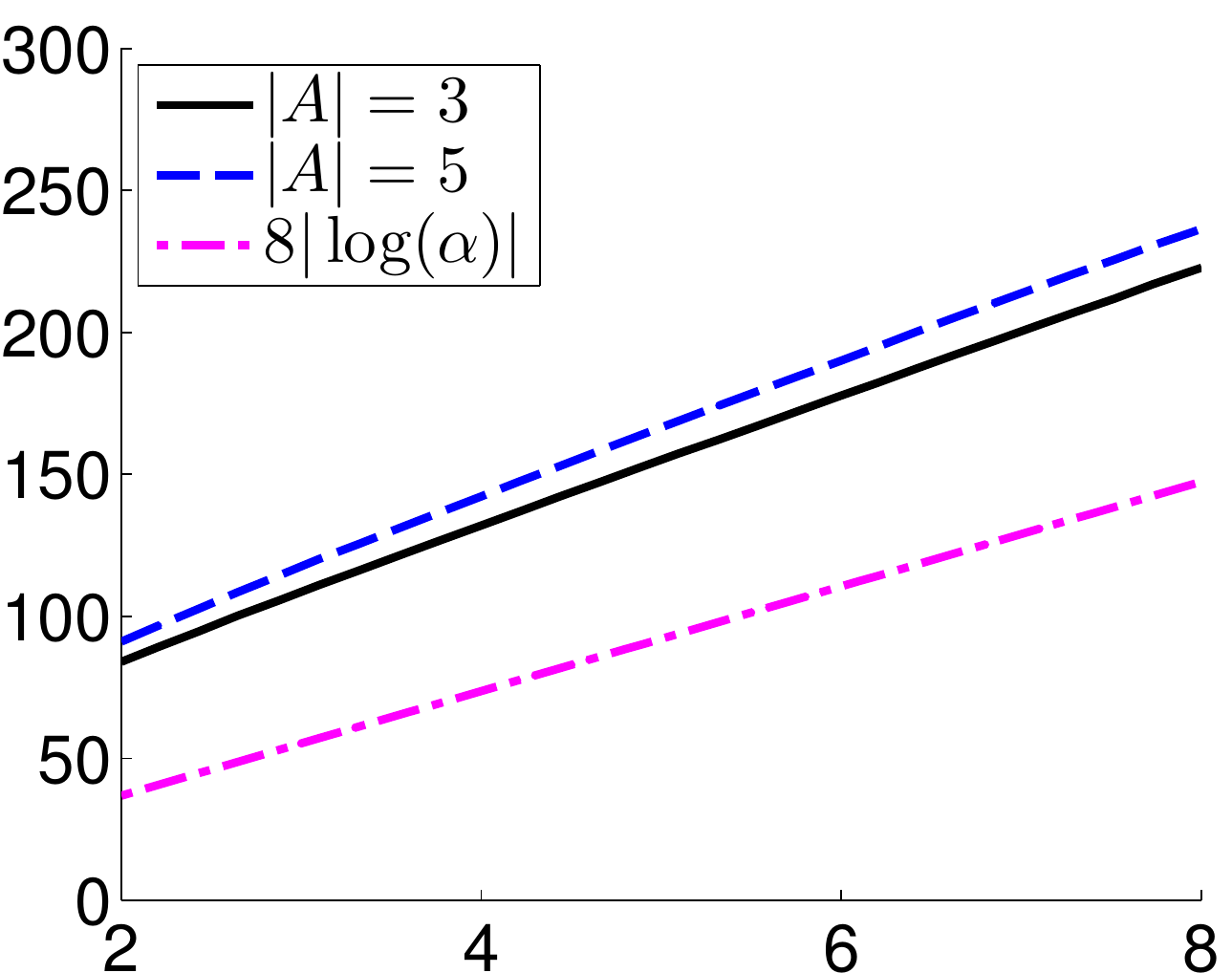}  
\label{fig:lu}
}\hspace{0.2in}
\subfloat[Normalized by $8|\log(\alpha)|$]{
\includegraphics[width=0.44\linewidth]{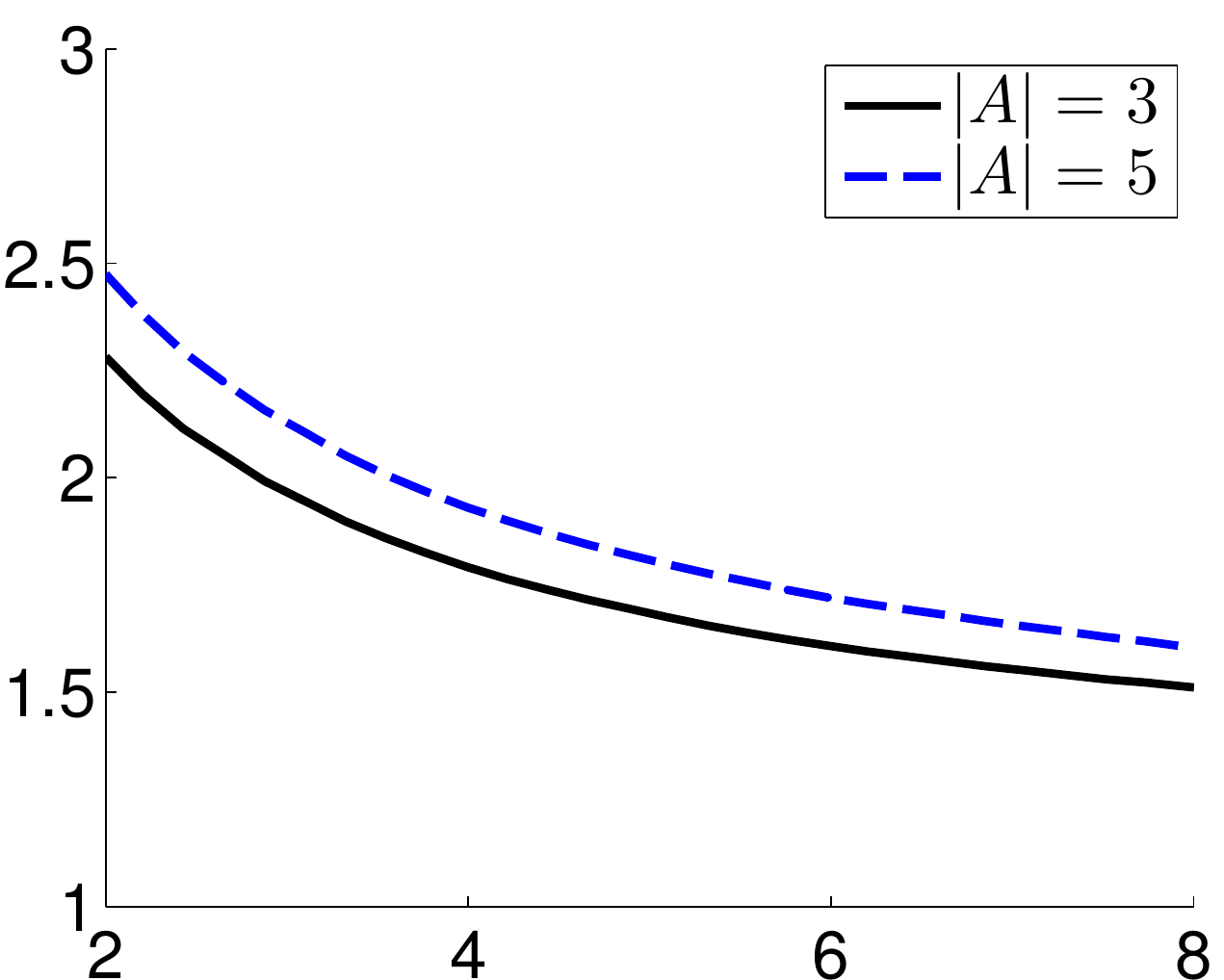}  
\label{fig:lu_n}
}\\
\subfloat[Intersection vs Incomplete: $\cP = \cP_{0,10}$]{
\includegraphics[width=0.44\linewidth]{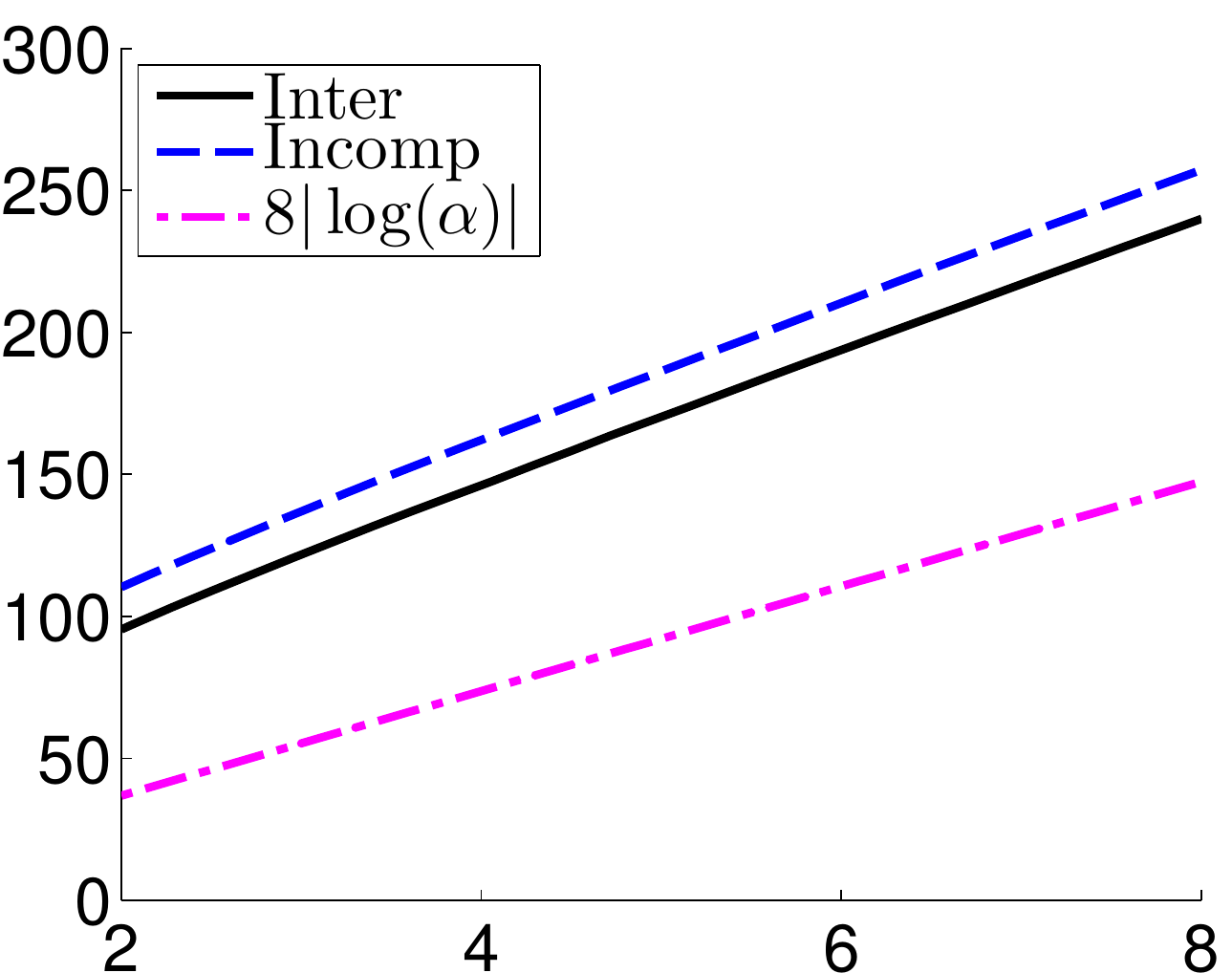}  
\label{fig:no_prior}
}\hspace{0.2in}
\subfloat[Normalized by $8|\log(\alpha)|$]{
\includegraphics[width=0.44\linewidth]{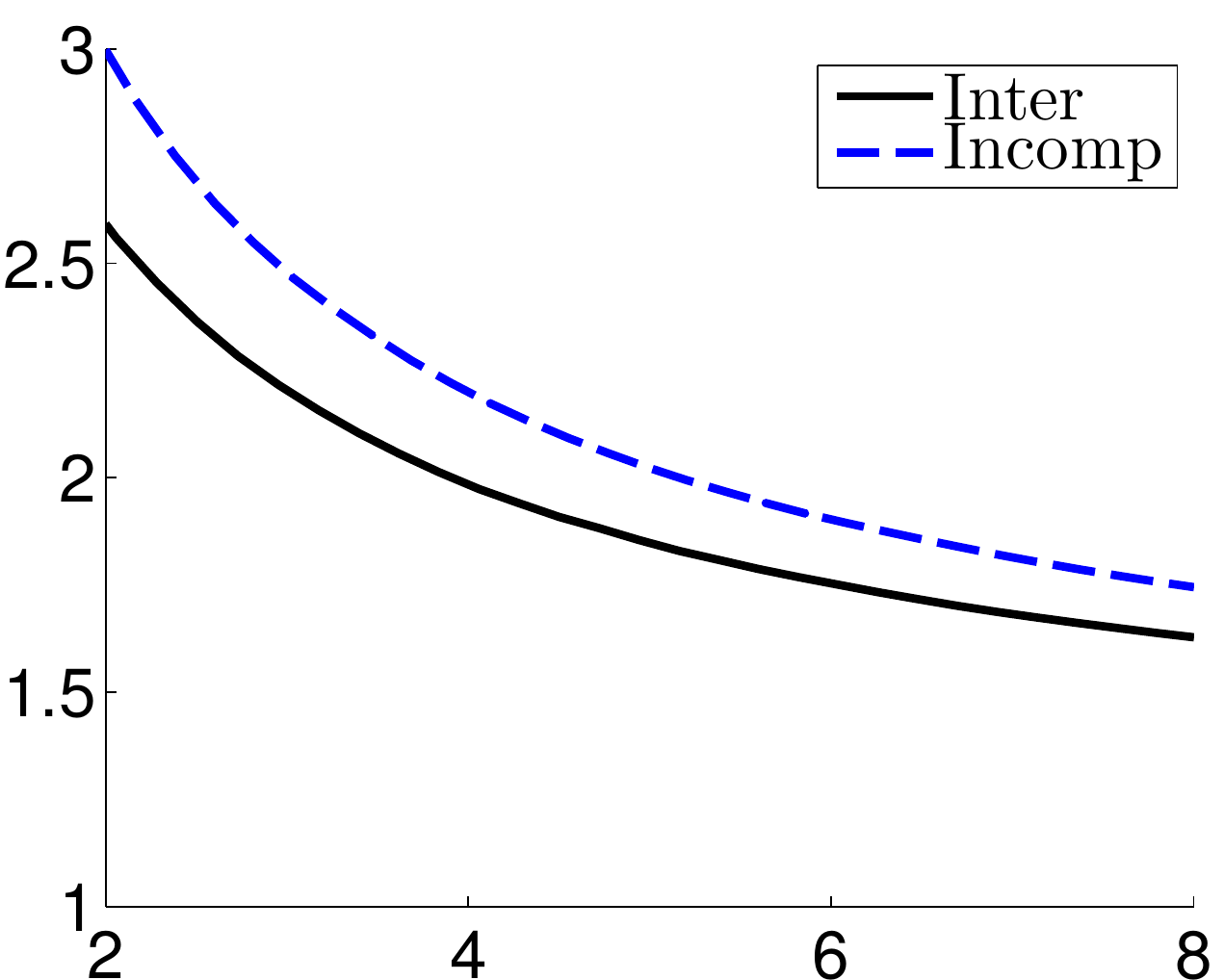}  
\label{fig:no_prior_n}
}
\captionsetup{justification=raggedright, singlelinecheck=false}
 \caption{ The x-axis in all graphs  is  $|\log_{10}(\alpha)|$.  In the first column, the y-axis denotes the expected sample size under $\Pro_\cA$ that is required in order to control the \textit{maximal} familywise type I error probability \textit{exactly} at level $\alpha$. 
 The dash-dot lines in each plot correspond to the first-order approximation, which is also a lower bound, to the optimal expected sample size  
for the class $\Delta_{\alpha,\alpha}(\cP)$; due to symmetry, this lower bound does not depend on $|\cA|$ in each setup. In the second column, we normalize each curve by its corresponding lower bound.
}
\label{fig:three}
\end{figure}

\begin{table}[htbp]\label{tab}
\caption{The standard error of the estimate is included in the parenthesis. The upper bound is on the error control given by~\eqref{bound} for the first table and by~\eqref{bound2} for the second.}
\subfloat[$\cP = \cP_m$. $(T_G,d_G)$ with $c = 10$. ]{
    \begin{tabular}{|c|c|c|c|}
    \hline
    $m$     &$\Pro_{\cA}(d_{G} \neq \cA)$  & $\Exp_{\cA}(T_{G})$  & Upper bound  \\
    \hline
    1     & 5.041E-05 (3.101E-07) & 64.071 (0.157)  & 4.086E-4   \\
    \hline
    3     & 6.034E-05 (5.343E-07) & 78.386 (0.157)  & 9.534E-4 \\
    \hline
    5     & 6.145E-05 (5.859E-07) &  81.070 (0.156) & 1.135E-3 \\
    \hline
    \end{tabular}%
    \label{tab:gap}
}\\
\subfloat[$\cP = \cP_{3,7}$. $(T_{GI},d_{GI})$ with $b  = 10$.]{
    \begin{tabular}{|c|c|c|c|}
    \hline
    $|A|$   & $\Pro_{\cA}(\cA \lesssim d_{GI})$  & $\Exp_{\cA}(T_{GI})$ & Upper bound  \\
    \hline
    3     & 3.653E-05 (5.447E-07) &  142.173 (0.264) & 4.540E-04\\
    \hline
    4     & 3.144E-05 (2.189E-07) &  152.873 (0.264) & 4.281E-04\\
    \hline
    5     & 2.621E-05 (1.825E-07) &  152.895 (0.263) & 3.891E-04\\
    \hline
    7     & 3.104E-07 (1.340E-08) &  142.363 (0.270) & 2.724E-04\\
    \hline
    \end{tabular}%
    \label{tab:MI}
}
\end{table}

\begin{figure}
\subfloat[Gap rule: $\cP = \cP_3$]{
\includegraphics[width=0.45\linewidth]{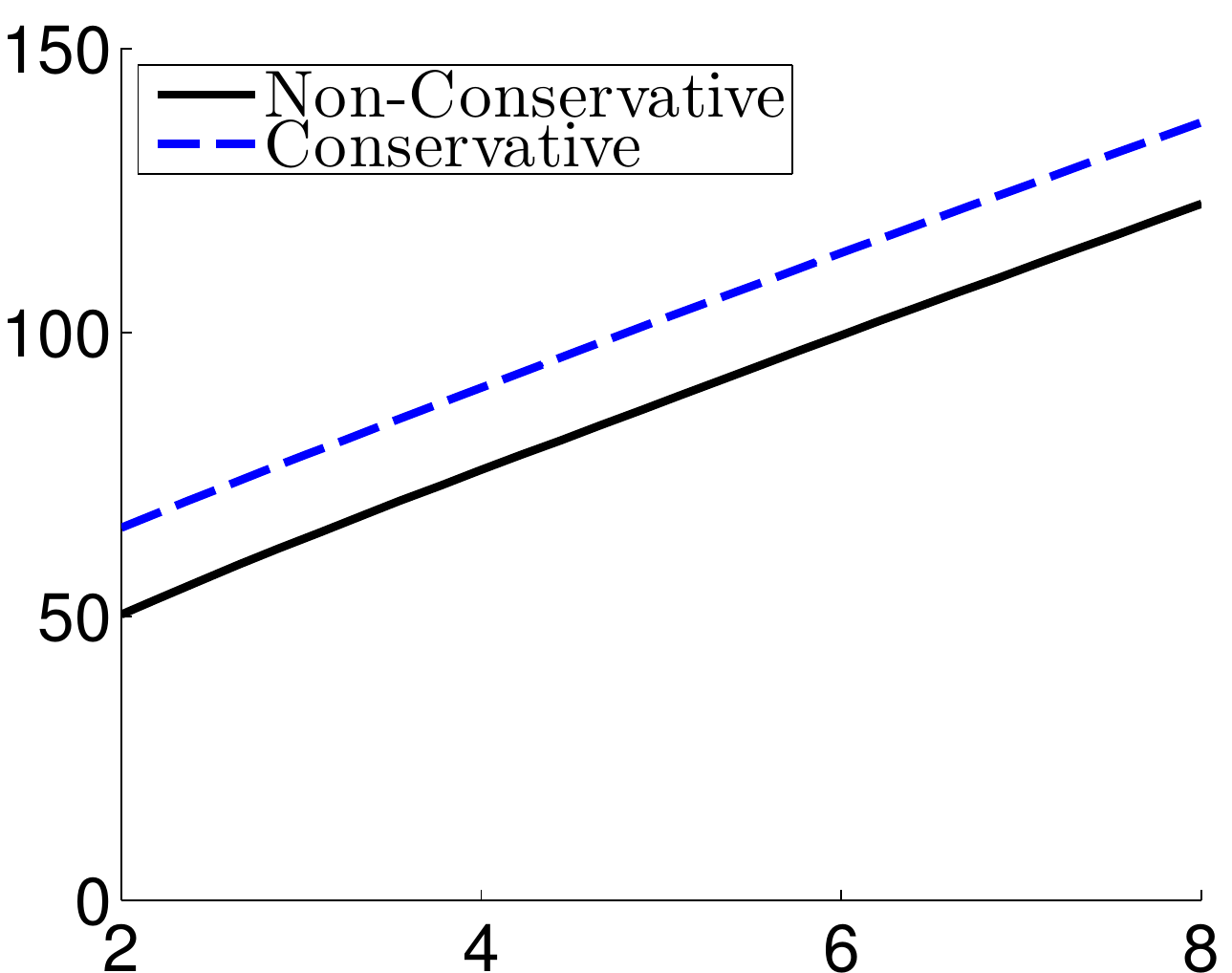}  
\label{fig:gap_upper}
}\hspace{0.2in}
\subfloat[Intersection rule: $\cP = \cP_{0,10}$]{
\includegraphics[width=0.44\linewidth]{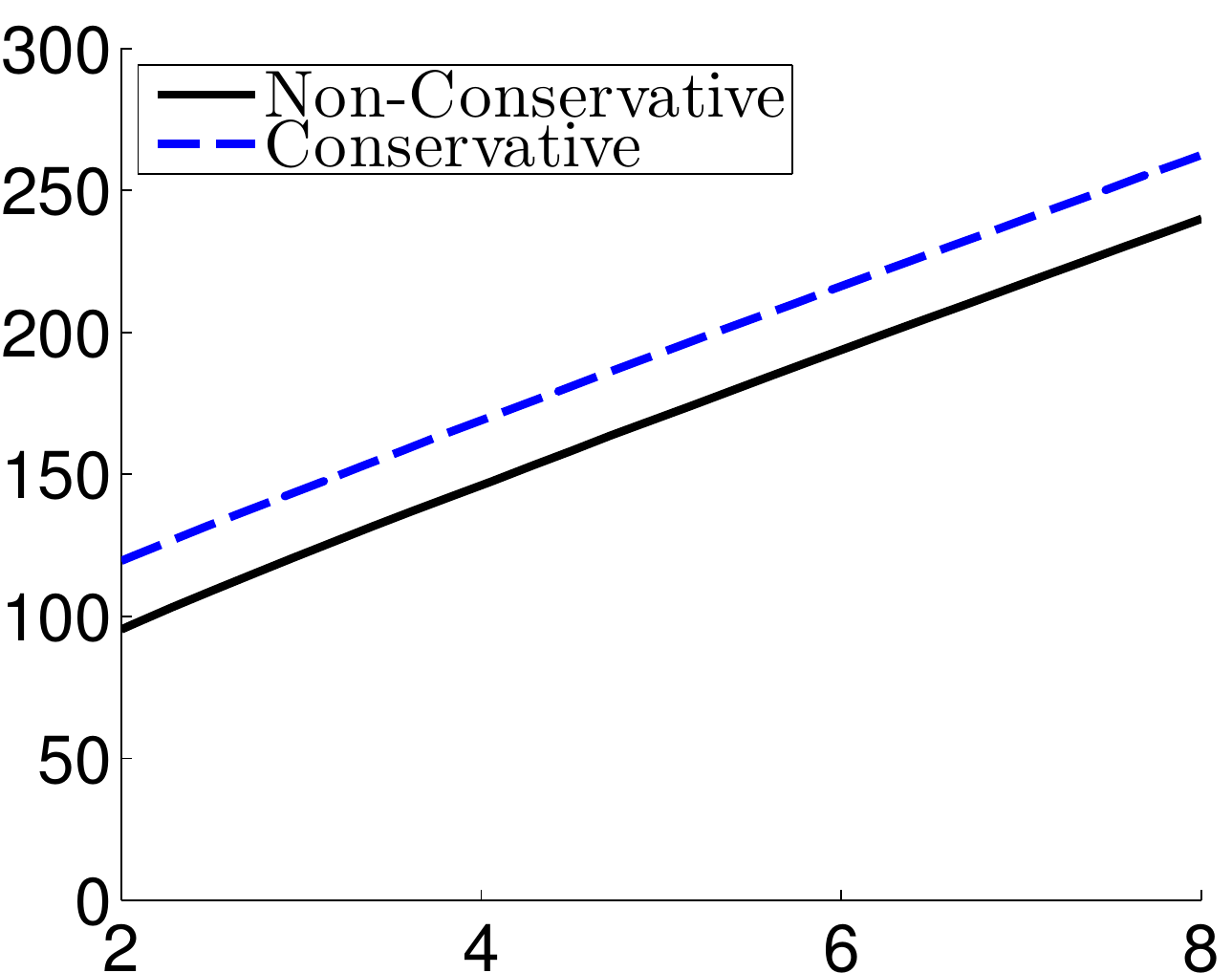}  
\label{fig:int_upper}
}
\caption{The x-axis is $|\log_{10}(\alpha)|$, where $\alpha$ is user-specified level. 
The y-axis is the expected sample size.
The dashed line uses the upper bound on the error probability to get conservative critical value, while the solid line uses the Monte Carlo approach to determine non-conservative threshold such that the \textit{maximal} familywise type I error is controlled \textit{exactly} at level $\alpha$.
}
\label{fig:upper}
\end{figure}

\section{Conclusions} \label{generalize}

We considered the problem of simultaneously testing multiple simple  null hypotheses, each of them against a simple alternative, in a sequential setup. That is, the data for each  testing problem are acquired sequentially and the goal is to stop sampling  as soon as possible, simultaneously in all streams, and make a correct  decision for each individual testing problem.  The main goal of this work was to  propose feasible, yet asymptotically optimal, procedures that  incorporate prior information on the number of signals (correct alternatives), and also to understand the potential gains in efficiency  by such prior information.

We studied this problem   under the assumption  that the  data streams for the various hypotheses are independent. Without any distributional assumptions on the data that are acquired in each stream,  we proposed procedures  that control the probabilities of at least one false positive and at least one false negative below arbitrary   user-specified levels. This was achieved in  two general  cases regarding the available prior information: when the exact  number of  signals is known in advance, and when we only have an upper and a lower bound for it. Furthermore, we proposed a Monte Carlo simulation method, based on importance sampling, that can facilitate the  specification of non-conservative critical values for the proposed multiple testing procedures in practice.  More importantly,  in the special case of i.i.d.  data in each stream, we were able to show that the proposed multiple testing procedures are asymptotically optimal, in the sense that they require the minimum possible expected sample size to a first-order asymptotic approximation as the error probabilities vanish at arbitrary rates. 

These  asymptotic optimality results have some interesting ramifications.
First of all, they imply that any refinements of the  proposed procedures,  for example using a more judicious choice of alpha-spending and beta-spending functions, cannot reduce the expected sample size  \textit{to a first-order} asymptotic  approximation. Second, they imply that bounds on the number of signals do not  improve the minimum possible expected sample size \textit{to a first-order  asymptotic approximation}, apart from a very special case.
On the other hand,  knowledge of the \textit{exact} number of signals  does reduce the minimum possible expected sample size to a first order approximation, roughly by a factor of 2. These insights were corroborated by a simulation study, which however also revealed the limitations of a first-order asymptotic analysis and emphasized the importance of second-order terms.

To our knowledge, these are the first results on the asymptotic optimality of  multiple testing procedures, with or without prior information, that control the familywise error probabilities of both types. However, there are still some important open questions that remain to be addressed. Do the proposed procedures attain, in the i.i.d. setup, the optimal expected sample size to a \textit{second-order} asymptotic approximation as well? 
Does the first-order asymptotic optimality property remain valid for more general, non-i.i.d. data in the streams? While we conjecture that the answer to both these questions is affirmative, we believe that the corresponding proofs require different  techniques from the ones we have used in the current paper.

There are also interesting generalizations  of the setup we considered in this paper. For example, it is interesting  to consider the sequential multiple testing problem when the goal is to  control generalized error rates, such as the false discovery rate~\citep{bartroff_arxiv}, instead of the more stringent  familywise error rates.  Another interesting direction is to allow the hypotheses in the streams to be specified up to an unknown parameter, or to consider a non-parametric setup similarly  to~\citet{li2014universal}. Finally, it is still an open problem to design asymptotically optimal multiple testing procedures that incorporate prior information on the number of signals when it is possible and desirable to stop sampling at different times in the various streams.

\appendix 
\section{Two lemmas} \label{appen}

\subsection{An information-theoretic inequality} \label{app_lower_bound_proof}

In the proof of  Theorem \ref{lower_bound} we use the  following, well-known,   information-theoretic  inequality, whose proof can be  found, e.g., in~\citet{tartakovsky2014sequential} (Chapter 3.2). 

\begin{lemma}\label{KL_bound}
Let $\Qro,\Pro$ be equivalent probability measures on a measurable space $(\Omega,\mathcal{G})$ and recall the function $\varphi$ defined in~\eqref{phi}. 
Then, for every $A \in \mathcal{G}$ we have
$$\Exp_{\Qro} \left[ \log \frac{d\Qro} {d \Pro} \right]  \geq \varphi \left( Q(A), \Pro(A^c) \right).$$ 

\end{lemma}


%
%

\subsection{A lemma on multiple random walks}
For the proof of  Lemmas \ref{ESS_TG} and \ref{ESS_LU} we need an upper bound on the expectation of the first time that  multiple random walks, not necessarily independent, are simultaneously above given thresholds. We state here the corresponding  result in some generality. 

Thus, let $L \geq 2$ and suppose that for each    $l \in [L]$ we have 
 a sequence of i.i.d. random variables, $\{\xi_n^l, n\in \bN\}$, such  that $\mu_l = \Exp [\xi_1^l] > 0$ and  $\text{Var} [\xi_1^l] < \infty$. For each    $l \in [L]$, let 
$$S_n^l = \sum_{i=1}^n \xi_i^l, \quad n \in \bN$$
be the corresponding random walk.  Here, \textit{no assumption is made on the dependence structure among these random walks}. For an arbitrary vector $(a_1,\ldots, a_L)$, consider the  stopping time 
\begin{equation*}
T = \inf\left\{ n \geq 1\;:\; S_n^l \geq a_l \text{ for every } l \in [L]\right\}.
\end{equation*}
The following lemma provides an upper bound on the expected value of $T$.
The proof is  identical to the one in Theorem 2 in~\citet{mei2008asymptotic}; thus we omit it. We stress that although the theorem in the reference assumes independent random walks, exactly the same proof applies to the case of dependent random  walks.

\begin{lemma}  \label{aux}
As $a_1,\ldots,a_L \to \infty$,
\begin{equation*}
\Exp[T] \leq  \max_{l \in [L]} \left( \frac{a_l}{\mu_l}  \right) +
 O\left(\sum_{l \in [L]}  \sqrt{  \frac{a_l}{\mu_l} } \right)  
\leq \max_{l \in [L]} \left( \frac{a_l}{\mu_l}   \right) +
 O \left( L  \sqrt{\max_{l \in [L] }  \{ a_l\}} \right)  .
\end{equation*}
\end{lemma}

\bibliographystyle{imsart-nameyear}
\bibliography{EJS_arxiv}

\end{document}